\documentclass{amsart}
\usepackage{amsmath,amssymb,mathrsfs,amsthm,delarray}
\usepackage{hyperref}
\numberwithin{equation}{section}

\newtheorem{theorem}{Theorem}[section]
\newtheorem{lemma}[theorem]{Lemma}
\newtheorem{corollary}[theorem]{Corollary}
\newtheorem{proposition}[theorem]{Proposition}

\newtheorem*{claim*}{Claim}

\newcommand{\C}{\mathbb C}
\theoremstyle{definition}
\newtheorem{definition}[theorem]{Definition}

\newtheorem*{acknowledgement}{Acknowledgement}
\theoremstyle{remark}
\newtheorem{remark}[theorem]{Remark}

\DeclareMathOperator{\Mod}{mod}

\title[Hybrid classes]{Complex analytic structure of hybrid classes of generalized polynomial-like mappings}
\author{Miaokeyan Lin}
\address{Shanghai Center for Mathematical Sciences, Fudan University}
\email{mkylin21@m.fudan.edu.cn}
\author{Weixiao Shen}
\address{Shanghai Center for Mathematical Sciences, Fudan University}
\email{wxshen@fudan.edu.cn}

\author{Yimin Wang}
\address{Department of Mathematics, Shanghai Normal University}
\email{yiminwang@shnu.edu.cn}
\keywords{hybrid class, vertical vector, horizontal vector}
\subjclass[2020]{Primary:37F20; Secondaries:37F31}
\date{\today}

\usepackage{svg}

\begin{document}
\maketitle
\begin{abstract}
In this paper, we study hybrid classes of generalized polynomial-like mappings. We show that locally they are complex analytic submanifolds with finite codimension equal to the total order of critical points in the filled Julia set (counting multiplicity). 
 \end{abstract}

\section{Introduction}
Polynomial-like mappings are fundamental objects in holomorphic dynamics. The notion was first introduced in~\cite{DH2} to study the self-similarity of the Mandelbrot set. Several generalized versions have since appeared in the literature, 
responding to the natural needs in studying high degree polynomials.


In~\cite{L2}, Lyubich endowed the space of quadratic-like mappings with a complex analytic structure and used this structure to establish the hyperbolicity of the Feigenbaum renormalization operator. A key ingredient in Lyubich's approach is to show that the hybrid class of a quadratic-like mapping with connected Julia set is a complex submanifold of codimension one in the whole space of quadratic-like mappings.  
The goal of this paper is to generalize the last result of Lyubich to the setting of generalized polynomial-like mappings, which we believe is useful in further developing renormalization theory for multicritical maps studied in~\cite{Smania, Wang}, among others.

A holomorphic mapping $f:U\to V$ is called a {\em generalized polynomial-like mapping (GPL)}, if 
\begin{itemize}
\item each of $U$ and $V$ is a non-empty finite union of Jordan disks in $\mathbb{C}$ with real analytic boundaries; 
\item $U\Subset V$;
\item for each component $U_0$ of $U$, $f|_{U_0}$ is a proper holomorphic map from $U_0$ onto a component of $V$;
\item each component of $V$ intersects $U$;
\item $f$ extends analytically to a neighborhood of $\overline{U}$ in $\C$ with no critical point in $\partial U$.
\end{itemize} 
The last condition is added for the convenience of exposition. 
The fourth condition implies that the {\em filled} Julia set of $f$ defined as 
\begin{equation}\label{eqn:Kf}
K(f)=\{z\in U: f^n(z)\in U, \forall n=0,1,\ldots\}
\end{equation}
is non-empty and intersects each component of $U$. 

We define an equivalence relation $\sim$ between GPLs: $f\sim \widetilde{f}$ if $K(f)=K(\widetilde{f})$ and $f=\widetilde{f}$ in a neighborhood of $K(f)$. 
An equivalence class in this relation is called a {\em GPL germ}. 
Let $\mathcal{GPL}$ denote the space of all GPL germs.

Similar to Lyubich's approach, we endow topology and complex structure on the space $\mathcal{GPL}$ using Banach injections 
\begin{equation}\label{eqn:jw}
j_W:\mathcal{Q}_W\to \mathcal{GPL}
\end{equation} 
defined below. 
For each bounded open subset $W$ of $\mathbb{C}$, let $\mathcal{B}_W$ denote the Banach space of continuous functions defined on $\overline{W}$ which are holomorphic in $W$, endowed with the supremum norm $\|\cdot\|_W$. For each $f\in \mathcal{B}_W$, let 
\begin{equation}\label{eqn:K*fW}
K_*(f)=K_*(f,W):=\{z\in \overline{W}: f^n(z)\in \overline{W}\text{ for all } n\ge 1\}. 
\end{equation}

Let $\mathcal{Q}_W$ be the collection of maps $f\in \mathcal{B}_W$ for which the following holds:
\begin{itemize}
\item $K_*(f)$ is non-empty and contained in $W$;
\item $f$ has a GPL restriction $f:U\to V$ with $K(f:U\to V)=K_*(f)$;
\item each component of $W$ intersects $K_*(f)$.
\end{itemize}
We define $j_W(f)$ to be the GPL germ determined by the GPL restriction $f:U\to V$. 
Clearly, each GPL $f:U\to V$ belongs to $\mathcal{Q}_U$ and is a representative of $j_U(f)$. 

\begin{proposition}\label{prop:cpxstr}
\begin{enumerate}
\item for each bounded open subset $W$ of $\C$, $\mathcal{Q}_W$ is an open subset of $\mathcal{B}_W$;
\item $j_W$ is injective; 
\item for bounded open subsets $W$ and $\widehat{W}$ of $\C$, if there is a non-empty open subset $\mathcal{U}_W$ of $\mathcal{Q}_W$ such that $j_W\mathcal{U}_W\subset j_{\widehat{W}} \mathcal{Q}_{\widehat{W}}$, then $\overline{W}\supset \widehat{W}$ and the transition map $j_{\widehat{W}}^{-1}\circ j_{W}: \mathcal{U}_W\to \mathcal{Q}_{\widehat{W}}$ is given by restriction $f\mapsto f|_{\widehat{W}}$. 
\end{enumerate}
\end{proposition}


We say that two generalized polynomial-like mappings $f,\widetilde{f}$ are {\em hybrid equivalent} if there is a qc map $\varphi:\mathbb{C}\to \mathbb{C}$ such that 
$\varphi(K(f))=K(\widetilde{f})$, 
\begin{equation}\label{eqn:conjugacy}
\varphi (f(z))=\widetilde{f}(\varphi(z))
\end{equation}
holds in a neighborhood of $K(f)$ and $\bar{\partial} \varphi=0$ holds a.e. on $K(f)$. This notion clearly descends to the level of GPL germs. The hybrid class of a GPL germ $f$ is the set of all GPL germs that are hybrid equivalent to $f$. 
For each GPL germ $f$, let $\mathcal{H}(f)$ denote the collection of all GPL germs which are hybrid equivalent to $f$. For $f\in \mathcal{Q}_W$ and any open subset $\mathcal{V}_W$ of $\mathcal{Q}_W$ containing $f$, let $\mathcal{H}_*(f,\mathcal{V}_W)$ denote the component of $j_W^{-1}(\mathcal{H}(f))\cap \mathcal{V}_W$  which contains $f$. 

The {\em criticality} of a GPL $f:U\to V$ is defined as 
\begin{equation}\label{eqn:criticality}
d(f):=\sum_{c\in \text{Crit}(f)\cap K(f)} \text{ord}(f,c),
\end{equation}
where $\text{ord}(f,c)=m$ if $f'(c)=f''(c)=\dots=f^{(m)}(c)=0$, but $f^{(m+1)}(c)\not=0$. 
The notion also descends to the level of GPL germs. 

The main theorem of this paper is the following:
\begin{theorem}\label{thm:hybridclass} Let $W_0$ be a bounded open subset of $\C$ and let $f_0\in \mathcal{Q}_{W_0}$. Then for any small neighborhood $\mathcal{V}$ of $f_0$ in $\mathcal{Q}_{W_0}$, the following holds: For each $f\in \mathcal{V}$ with $d(j_{W_0}f)=d(j_{W_0}f_0)$,  $\mathcal{H}_*(f,\mathcal{V})$ is an embedded complex analytic submanifold of $\mathcal{B}_{W_0}$ of codimension $d(j_{W_0} f_0)$ and closed in $\mathcal{V}$.    
\end{theorem}

Here, we say that a subset $E$ of a complex Banach space $\mathcal{B}$ is an embedded submanifold of codimension $d$ if for any $f\in E$, there is a holomorphic map $\psi:\mathcal{U}\to \C^d$, where $\mathcal{U}$ is a neighborhood of $f$ in $\mathcal{B}$, such that $\psi(f)=0$, $\psi^{-1}(0)=E\cap\mathcal{U}$, and such that 
$D\psi_g$ has rank $d$ for each $g\in \mathcal{U}$. For $f\in E$, we say that $E$ is a complex analytic submanifold of codimension $d$ near $f$, if there is a neighborhood $\mathcal{V}$ of $f$ in $\mathcal{B}$ such that $E\cap\mathcal{V}$ is a complex analytic submanifold of $\mathcal{B}$  of codimension $d$. 
   
In general, the set $j_W^{-1}(\mathcal{H}(j_W f))\cap\mathcal{V}_W$ may have infinitely many connected components, see~\cite{Branner}. It would be interesting to know whether this set is closed in $\mathcal{V}_W$.

The theorem was proved by Lyubich~\cite[Section 4]{L2} in the quadratic-like case, i.e. in the case $U$ and $V$ are both connected and $f_0$ has a unique critical point which is non-degenerate. Lyubich's proof relies on the continuity of the straightening map in the quadratic case; this continuity, combined with the Ma\~n\'e–Sad–Sullivan theorem, allows one to reduce the problem to the (easy) case where $f_0$ 
is structurally stable.  However, as shown in~\cite[Chapter 3]{DH2}, the straightening map is in general discontinuous for high degree polynomial-like maps, so we have to use a different argument. 

Our strategy is as follows. As is well-known, we can equip a small neighborhood of a given map with a holomorphic motion which encodes the holomorphic dependence of the boundary dynamics of specified representatives of the GPL germs. This helps us to `identify' a connected component of a local hybrid class with a component of a marked hybrid class, which belongs to the fixed-point set of the `mating' map $I^f$.  Formally, $I^f$ is idempotent, so if it were a self-map of a Banach manifold, the submanifold property would follow from a standard retraction argument. 
But it is actually defined on an open set in $\mathcal{B}_W$  
and takes values in $\mathcal{B}_{\widehat{W}}$ with
$\widehat{W}\Subset W$, hence cannot be composed with itself. We circumvent this by showing that the `tangent' space of the fixed-point set at $f$ is dense in the horizontal space $E^h(f)$, which forces the fixed-point set to be a graph over $E^h(f)$. 

A significant part of our proof is to show that the tangent spaces (horizontal) of the leaves (the complex submanifolds asserted in the theorem) have codimension $d(j_Wf)$.  This will be done by showing first that the codimension is constant along a curve with fixed criticality using a straightforward extension of Lyubich's argument in the quadratic-like case, together with Astorg's result~\cite{Astorg} on the Teichm\"uller embedding in the space of polynomials. A bridge to connect these results is an extension of Baranski's result~\cite[Corollary 9]{Baranski} on straightening of GPL maps, where we removed an extra condition. Beyond the quadratic-like setting, the known results, to the best of our knowledge, are restricted to maps without invariant line field on the filled Julia set; see ~\cite{ALM, Smania}.

The rest of the paper is organized as follows. In \S~\ref{sec:space}, we prove Proposition~\ref{prop:cpxstr}. In \S~\ref{sec:markedclass}, we introduce marked hybrid class, and show that each connected component $\mathcal{H}_*(f,\mathcal{V})$ lies in such a class. In \S~\ref{sec:lyubichdecomp}, we extend the Lyubich horizontal-vertical decomposition of the tangent space to the GPL setting. In \S~\ref{sec:codim}, we study the codimension of the horizontal space.  In \S~\ref{sec:hybridclass}, we use the `mating' map to prove Theorem~\ref{thm:hybridclass}. 
\section{Complex structure of the space of GPL germs}\label{sec:space}
This section is devoted to the proof of Proposition~\ref{prop:cpxstr}.
For each bounded open subset $W$ of $\C$, let
$$\mathcal{P}_W=\{f\in \mathcal{B}_W: K_*(f)\subset W\},$$
where $K_*(f)=K_*(f,W)$ is as in (\ref{eqn:K*fW}). 

\begin{lemma} \label{lem:K*cont}
\begin{enumerate}
\item If $f_n\to f$ in $\mathcal{B}_W$, $z_n\in K_*(f_n)$ and $z_n\to z_0$, then 
$z_0\in K_*(f)$.
\item The set $\mathcal{P}_W$ is open in $\mathcal{B}_W$.
\end{enumerate}
\end{lemma}
\begin{proof}
(1) For any $k=1,2,\dots$ and any $n$,
\begin{equation}\label{eqn:fnkzn}
f_n^k(z_n)\in \overline{W}.
\end{equation}
For each fixed $k$, letting $n\to\infty$ in (\ref{eqn:fnkzn}), we obtain 
$f^k(z_0)\in \overline{W}$. Thus $z_0\in K_*(f)$.

(2) Arguing by contradiction, assume that there exists $f\in \mathcal{P}_W$ which is not in the interior of $\mathcal{P}_W$. Then there exists a sequence $\{f_n\}_{n=1}^\infty \subset \mathcal{B}_W$ which converges to $f$ such that for each $n=1,2,\ldots$, there exists $z_n\in K_*(f_n)\cap \partial W$. 
Passing to a subsequence, we may assume that $z_n\to z_0$. 
By (1), $z_0\in K_*(f)$. Since $z_0\in \partial W$, $K_*(f)\not\subset W$, which contradicts the fact $f\in \mathcal{P}_{W}$. 
\end{proof}

\begin{lemma} \label{lem:smallglp}
Let $f\in \mathcal{Q}_W$. Then there is a GPL $f:U\to V$ such that the following hold:
\begin{itemize} 
\item $K(f:U\to V)=K_*(f)$; 
\item $U\Subset W$; 
\item $f$ has no critical point in $\overline{U}\setminus K_*(f)$. 
\end{itemize}
\end{lemma}
\begin{proof} By definition, $f$ has a GPL restriction $f:V'\to V$ with $K(f:V'\to V)=K_*(f)$. Define $V_0=V$, 
and for each $k\ge 1$, define inductively $V_k$ to be the union of components of $f^{-1}(V_{k-1})\cap V'$. Then $K(f:V_k\to V_{k-1})=K(f:V'\to V)=K_*(f)$ and $\overline{V}_k$ is a descending sequence of compact sets with $\bigcap_{k=0}^\infty \overline{V_k}=K_*(f)$. Take $k$ sufficiently large such that $\overline{V}_k\setminus K_*(f)$ contains no critical point of $f$ and such that $\overline{V}_k\Subset W$. Let $U=V_k$ and $V=V_{k-1}$. 
\end{proof}
Given a domain $\mathcal{V}$ in a complex Banach manifold $\mathcal{B}$ with a base point $\ast\in \mathcal{V}$ and a subset $E$ of $\C$, a holomorphic motion of $E$ over $(\mathcal{V},0)$ is a family of injections $h_\lambda:E\to \C$, $\lambda\in \mathcal{V}$, such that $h_*=id_E$, $z\mapsto h_\lambda(z)$ is holomorphic in $\lambda$ for each $z\in E$. By Ma\~n\'e-Sad-Sullivan's $\lambda$-lemma~\cite{MSS}, $(\lambda,z)\mapsto h_\lambda(z)$ is continuous. By a theorem due to Bers-Royden~\cite{BR}, if $\mathcal{V}=\mathcal{B}_r$, then a holomorphic motion of $E$ over $(\mathcal{B}_r,\ast)$ extends to a holomorphic motion of the entire complex plane $\C$ over $(\mathcal{B}_{r/3},\ast)$, where $\mathcal{B}_r$ denotes the open ball in $\mathcal{B}$ centere at $\ast$ and of radius $r$. Furthermore, $h_\lambda$ is a holomorphic motion of $\C$ over $(\mathcal{B}_r,\ast)$, then each $h_\lambda$ is a qc homeomorphism of $\C$, the complex dilatation $\mu_\lambda$ depends holomorphically on $\lambda$, and $\|\mu_\lambda\|_\infty\le |\lambda|/r.$

\begin{lemma}\label{lem:persistglp} Let $f\in \mathcal{Q}_W$ and let $f:U\to V$ be a GPL restriction as in Lemma~\ref{lem:smallglp}. Then there is a neighborhood $\mathcal{V}_W$ of $f$ in $\mathcal{B}_W$,  and a holomorphic motion $h_g$ of $\C$ defined over $(\mathcal{V}_W, f)$ such that 
\begin{equation}\label{eqn:ghg}
g(h_g(z))=f(z), \,\, z\in \partial U, g\in \mathcal{V}_W.
\end{equation}
\end{lemma}
\begin{proof} It suffices to construct a holomorphic motion $h_g$ of $\partial U$ over some neighborhood $\mathcal{V}_W$ of $f$ in $\mathcal{B}_W$ such that (\ref{eqn:ghg}).  Once this done, shrinking $\mathcal{V}_W$ and define $h_g|_{\C\setminus V}=id$, we obtain a holomorphic motion of $(\C\setminus V)\cup \partial U$. Applying Bers-Royden's result gives us the desired holomorphic motion over an even smaller $\mathcal{V}_W$.

Since $\partial U\subset W$ is compact and contains no critical point of $f$, there exists $r>0$ such that 
\begin{itemize}
\item $K':=\bigcup_{z\in \partial U} \overline{B(z,r)}$ is a compact subset of $W$; 
\item for each $z\in \partial U$, $f$ is univalent in $B(z,r)$ and
\item  $$\varepsilon_1:=\inf_{|\zeta-z|\le r} |f'(\zeta)|>0, \, \varepsilon_2:=\inf_{|\zeta-z|=r} |f(z)-f(\zeta)|>0.$$ 
\end{itemize}
Let $\mathcal{V}_W$ be the $\varepsilon$-neighborhood of $f$ in $\mathcal{B}_W$, where $\varepsilon<\varepsilon_2$ is a small positive constant. Then for any $g\in \mathcal{V}_W$, $|g(\zeta)-f(\zeta)|<\varepsilon_2$ holds for any $\zeta\in \partial B(z, r)$. 
By Rouch\'e's theorem, $g(w)=f(z)$ has a unique solution in $B(z,r)$, denoted by $h_g(z)$. When $\varepsilon$ is taken small enough, $\sup_{\zeta\in K'}|g'(\zeta)-f'(\zeta)|<\varepsilon_1$, so $g'(\zeta)\not=0$ in $K'$. By the Implicit Function Theorem, $g\mapsto h_g(z)$ is holomorphic in $g$. 

To complete the proof, it remains to show that $h_g|_{\partial U}$ is injective in for each $g\in \mathcal{V}_W$. To this end, it suffices to show that for any distinct $z_1,z_2\in \partial U$, $E=\{g\in \mathcal{V}_W: h_g(z_1)=h_g(z_2)\}$ is open. Indeed, since $E$ is clearly closed in $\mathcal{V}_W$ and $\mathcal{V}_W$ is connected, the openness of $E$ would imply $E=\mathcal{V}_W$ or $\emptyset$. The former case does not happen since $f\not\in E$. To show that $E$ is open, suppose $g_0\in E$. Since $g(h_{g}(z_i))=f(z_i)$, we have $w_0:=f(z_1)=f(z_2)$ and $h_g(z_1), h_g(z_2)$ are both solutions of the equation $g(\zeta)=w_0$. When $g$ is close to $g_0$, $h_g(z_i)$ are close to $\zeta_0:=h_{g_0}(z_1)$. Since $g_0'(\zeta_0)\not=0$, it follows that $h_g(z_1)=h_g(z_2)$ when $g$ is sufficiently close to $g_0$. This shows that $g_0$ is in the interior of $\mathcal{V}_W$. 
\end{proof}
\begin{proof} [Proof of Proposition~\ref{prop:cpxstr}]
(1) Let $f\in \mathcal{Q}_W$ and let $f:U\to V$ be a GPL restriction given by Lemma~\ref{lem:smallglp} and let $h_g$ be the holomorphic motion of $\partial U$ over $(\mathcal{V}_W, f)$ given by Lemma~\ref{lem:persistglp}. Let $U_g=h_g(U)$. Then $g: U_g\to V$ is a GLP. Since $K_*(f)\subset U$, by Lemma~\ref{lem:K*cont}, shrinking $\mathcal{V}_W$ if necessary, we may assume that for $g\in \mathcal{V}_W$, $K_*(g)\subset U_g$. It follows that $K_*(g)\subset K(g:U_g\to V)$ and hence these two sets are equal. 
%

(2) Suppose that $j_Wf_1=j_Wf_2$ holds for some $f_1, f_2\in \mathcal{Q}_W$. Then $f_1=f_2$ holds in a neighborhood $W'$ of $K_*(f_1)=K_*(f_2)$. Since $W'$ intersects each component of $W$, it follows that $f_1=f_2$ in $W$. 
 
(3) Arguing by contradiction, assume that 
there exists $a\in  \widehat{W}\setminus \overline{W}$.
 
For each $g\in \mathcal{V}_W$, let $\widehat{g}\in \mathcal{Q}_{\widehat{W}}$ be such that $j_{\widehat{W}}\widehat{g}=j_W g$. 
Take $f\in \mathcal{V}_W$. Since each component of $W$ intersects the compact subset $K_*(f,W)$, the number of connected components of $W$ is finite. Similarly, the number of components of $\widehat{W}$ is also finite. So there exists $r>0$ such that 
\begin{enumerate}
  \item[(i)] for each $z\in K_*(f,W)=K_*(\widehat{f},\widehat{W})$, $B(z, 2r)\subset W\cap \widehat{W}$.
  \item[(ii)] if $z_1, z_2\in K_*(f)$ belong to different components of $W$, then $|z_1-z_2|>2r$. 
\end{enumerate}
By Lemma~\ref{lem:K*cont}, there exists a small neighborhood $\mathcal{U}_W$ of $f$ in $\mathcal{V}_W$ such that 
\begin{enumerate}
\item [(iii)] for any $g\in  \mathcal{U}_W$, 
$K_*(g)$ is contained in the $r$-neighborhood of $K_*(f).$
\end{enumerate}
Consider $f_t(z)=f(z)+t/(z-a)$. Let $\varepsilon>0$ be a constant such that $f_t\in \mathcal{U}_W$ when $|t|<\varepsilon$. 

Fix $t$ with $0\le |t|<\varepsilon.$ We shall show that $(\widehat{f}_t-\widehat{f})(z)=t/(z-a)$ holds for all $z$ in each component $\widehat{U}$ of $\widehat{W}$, which is a contradiction since $t/(z-a)$ is not holomorphic in $\widehat{W}$.  
Indeed, by definition of $\mathcal{Q}_{\widehat{W}}$, there exists $\widehat{z}_0\in K_*(\widehat{f}_t,\widehat{W})\cap \widehat{U}$. By (iii), there exists $z_0\in K_*(f,W)$ such that $\widehat{z}_0\in B(z_0,r)$. 
By our choice of $r$, $B(z_0,r)\subset \widehat{U}\cap W$. 
Since $f_t=\widehat{f}_t$ holds in a neighborhood of $\widehat{z}_0\in K_*(f_t,W)$, it holds in $B(z_0,r)$ as well. Since $f=\widehat{f}$ in a neighborhood of $z_0\in K(f,W)$, it also holds on $B(z_0,r)$. It follows that $\widehat{f}_t(z)-\widehat{f}(z)=f_t(z)-f(z)=t/(z-a)$ holds in $B(z_0,r)$. Therefore, 
$\widehat{f}_t(z)-\widehat{f}(z)=t/(z-a)$ holds in $\widehat{U}$.  
\end{proof}

\section{Equipped Banach neighborhoods and marked hybrid classes}\label{sec:markedclass}
In this section, we define an equipped Banach neighborhood, which, roughly speaking, means that an open set in $\mathcal{Q}_W$ where we can choose GPL representatives $f:U_f\to V_f$ such that the boundary dynamics moves holomorphically. The main result is Proposition~\ref{prop:connected set}, which asserts that in such a neighborhood, each connected set in the hybrid class is contained in a marked hybrid class. 

We start with the precise definition of an equipped Banach neighborhood. 
\begin{definition}\label{dfn:equipped}
A non-empty open subset $\mathcal{V}_W$ of $\mathcal{Q}_W$ is called {\em equipped} if there is a holomorphic motion $h_f$ of $\C$ based on $(\mathcal{V}_W, f_0)$ for some $f_0\in \mathcal{V}_W$, and for each $f\in \mathcal{V}_W$, $j_Wf$ has a GPL representative $f: U_f\to V_f$ such that 
\begin{itemize}
\item $f_0$ has no critical point in $\overline{U}_{f_0}\setminus K_*(f_0)$;
\item $h_f(U_{f_0})=U_f$ and $h_f(V_{f_0})=V_f$;
\item $h_f\circ f_0=f\circ h_f$ on $\partial U_{f_0}$;
\item there is a compact subset $L$ of $\C$ such that $L\supset V_f$ and $h_f|_{\C\setminus L}=id$.
\end{itemize}   
\end{definition}
We say that such a $\mathcal{V}_W$ is an equipped neighborhood of $f_0\in \mathcal{Q}_W$. 
By Lemmas~\ref{lem:smallglp} and~\ref{lem:persistglp}, for each $f_0\in \mathcal{Q}_W$, a sufficiently small neighborhood $\mathcal{V}_{W}$ of $f_0|_{W}$ can be equipped. We can even choose the holomorphic motion $h_f$ such that $h_f|_{\C\setminus V_{f_0}}=id$. 
Note that for any $f\in \mathcal{V}_{W}$ with $d(j_Wf)=d(j_W f_0)$,
$\mathcal{V}_W$ is also an equipped neighborhood of $f|_W$. 

Given an equipped neighborhood $\mathcal{V}_W$ of $f_0$ as above, for each $f\in \mathcal{V}_W$, we define $\mathcal{H}_\#(f, \mathcal{V}_W)$, called the {\em marked hybrid class} of $f$, to be the collection of $g\in \mathcal{V}_W$ for which $h_g\circ h_{f}^{-1}|_{\C\setminus U_{f}}$ extends to a qc map $\psi:\mathbb{C}\to\mathbb{C}$ such that 
\begin{itemize}
\item $\bar{\partial} \psi=0$, a.e. on $K_*(f)$;
\item $\psi\circ f=g\circ \psi$ holds on $U_{f}$. 
\end{itemize}

Let us also introduce the following notation:
$$\mathcal{H}(f,\mathcal{V}_W)=\{g\in \mathcal{V}_W: j_Wg \mbox{ is hybrid equivalent to } j_Wf\},$$
and $\mathcal{H}_*(f,\mathcal{V}_W)$ is the connected component of $\mathcal{H}(f,\mathcal{V}_W)$ which contains $f$. 

\begin{proposition}\label{prop:connected set} 
Let $\mathcal{V}_W$ be a neighborhood of $f_0\in \mathcal{Q}_W$ which is equipped with a holomorphic motion $h_f$ as in Definition~\ref{dfn:equipped}. Then for $f\in \mathcal{V}_W$ with $d(j_W f)=d(j_W f_0)$, 
$\mathcal{H}_*(f,\mathcal{V}_W)$ is the connected component of $\mathcal{H}_\#(f, \mathcal{V}_W)$ which contains $f$. 
\end{proposition}

We need a few lemmas for the proof of the proposition. 

\begin{lemma}\label{lem:closedH}
Let $\mathcal{V}_W$ be an equipped neighborhood of $f_0\in \mathcal{Q}_W$. Then 
$\mathcal{H}_\#(f_0,\mathcal{V}_W)$ is a closed subset of $\mathcal{V}_W$.
\end{lemma}
\begin{proof} Suppose that $g_n\in \mathcal{H}_\# (f_0,\mathcal{V}_W)$ converges in $\mathcal{B}_W$ to $g_\infty\in \mathcal{V}_W$. We need to show that $g_\infty\in \mathcal{H}_\# (f_0,W)$. To this end, we shall prove that there exists a sequence of qc maps $\theta_n:\C\to \C$ which conjugate 
$f_0:U_{f_0}\to V_{f_0}$ to $g_\infty:U_{g_\infty}\to V_{g_\infty}$ and coincide with $h_{g_\infty}\circ h_{f_0}^{-1}$ outside $K_*(f_0)$ such that the $L^\infty$ norm of the Beltrami coefficient of $\theta_n$ on $K_*(f_0)$ converges to zero. 

Assume for the moment that such a sequence $\theta_n$ exists. Let us show that $g_\infty\in \mathcal{H}_\#(f_0,W)$. Passing to a subsequence, we may assume that $\theta_n$ converges to a qc conjugacy $\theta$ from $f_0:U_{f_0}\to V_{f_0}$ to $g_\infty:U_{g_\infty}\to V_{g_\infty}$. We need to show that $\bar{\partial} \theta=0$ a.e. on $K_*(f_0)$. To this end, let $\tau:\C\to\C$ be the qc map with $\bar{\partial}\tau=\bar{\partial }\theta$ outside $K_*(f_0)$ ad $\bar{\partial }\tau=0$ a.e. on $K_*(f_0)$. The qc map $\theta_n\circ \tau^{-1}$ has maximal dilatation converges to $1$, since $\theta_n=\theta$ outside $\tau(K_*(f_0))$. Thus their limit $\theta\circ\tau^{-1}$ is conformal, which implies that $\bar{\partial}\theta=0$ a.e. on $K_*(f_0)$.  

Now we turn to the proof of existence of $\theta_n$. By ~\cite{BR}, the complex dilatation of $h_f$ depends on $f\in \mathcal{V}_W$ holomorphically. So there exists $C_n\to 1$ such that for $m\ge n$, 
the maximal dilatation of $h_{g_m}\circ h_{g_n}^{-1}$ does not exceed $C_n$. By assumption, $h_{g_m}\circ h_{g_n}^{-1}|_{\C\setminus U_{g_n}}$ extends to a qc conjugacy $\psi_{m,n}$ from $g_n:U_{g_n}\to V_{g_n}$ to $g_m:U_{g_m}\to V_{g_m}$ with $\bar{\partial} \psi_{m,n}=0$ a.e. on $K_*(g_n)$. The maximal dilatation of $\psi_{m,n}$ is at most $C_n$. 
For each $n$, there is a sequence $m_k\to\infty$ such that $\psi_{m_k,n}$ converges to a $C_n$-qc map $\psi_{\infty, n}$ which conjugates $g_n: U_{g_n}\to V_{g_n}$ to $g_\infty: U_{g_{\infty}} \to V_{g_{\infty}}$. Moreover, $\psi_{\infty,n}=h_{g_\infty}\circ h_{g_n}^{-1}$ outside $U_{g_n}$. 
On the other hand, $h_{g_n}\circ h_{f_0}^{-1}|_{\C\setminus U_{f_0}}$ extends to a qc conjugacy $\varphi_n$ from $f_0:U_{f_0}\to V_{f_0}$ to $g_n: U_{g_n}\to V_{g_n}$ and $\bar{\partial} \varphi_n=0$ a.e. on $K_*(f_0)$.  Thus $\theta_n:=\psi_{\infty,n}\circ \varphi_n$ is a qc conjugacy from $f_0:U_{f_0}\to V_{f_0}$ to $g_\infty: U_{g_\infty}\to V_{g_\infty}$. Each $\theta_n$ coincide with $h_{g_\infty}\circ h_{f_0}^{-1}$ outside $U_{f_0}$, hence they coincide outside $K_*(f_0)$. The essential supremum of the complex dilatation of $\theta_n$ on $K_*(f_0)$ is bounded from above by $C_n\to 1$.
\end{proof}

The following two lemmas are based on the well-known pull-back argument. 
\begin{lemma}\label{lem:hgextension}
Let $\mathcal{V}_W$ be an equipped neighborhood of $f_0\in \mathcal{Q}_W$. Then for each $g\in \mathcal{H}_*(f_0,\mathcal{V}_W)$, $h_g|_{\C\setminus U_{f_0}}$ extends to a qc map $H_g:\C\setminus K_*(f_0)\to \C\setminus K_*(g)$ such that $g\circ H_g=H_g\circ f_0$ holds on $U_{f_0}\setminus K_*(f_0)$.
\end{lemma}
\begin{proof} Let $E=\mathcal{H}_*(f_0,\mathcal{V}_W)$. The connectivity of $E$ is essential in the following proof which uses the well-known pull-back argument. 

For each $g\in E$, let $V^0_g=V_g$ and for each $n\ge 1$, denote $V^n_g=(g|U_g)^{-1}(V_g^{n-1})$. We claim that for each $n\ge 1$, for each $g\in E$, there is a qc map $H_{g,n}:\C\to \C$ such that 
$g\circ H_{g,n}(z)=H_{g,n}(f_0(z))$ for all $z\in U_{f_0}\setminus V^n_{f_0}$ and such that $H_{g,n}=h_g$ on $\C\setminus U_{f_0}$. 

Indeed, for $n=1$, we just take $H_{g,1}=h_g$. Let $n\ge 1$ and assume that the maps $H_{g,n}$ have been constructed for all $g\in E$. We construct $H_{g, n+1}$ as follows. For each $z\not\in V^n_{f_0}$, let $H_{g, n+1}(z)=H_{g,n}(z)$. Now consider a component $P$ of $V_{f_0}^n$. Let $Q=f_0(P)$. Moreover, let $P_g=H_{g,n}(P)$ and $Q_g=H_{g,n}(Q)$. Note that $f_0:P\to Q$ and $g: P_g\to Q_g$ have the same degree since the boundaries maps are conjugacy by the homeomorphism $H_{g,n}$ and hence have the same degree.  
\begin{itemize}
\item If $f_0: P\to Q$ is conformal, we simply define $H_{g,{n+1}}|_P=(g|_{P_g})^{-1}\circ H_{g,n}\circ f_0$ on $P$.
\item Assume that $f_0:P\to Q$ has a critical point. Let $CV_g$ denote the set of critical values of $g:P_g\to Q_g$. In the Hausdorff topology, $CV_g$ moves continuously in $g\in E$. Since $d(j_Wg)$ is constant for $g\in E$, $CV_g\cap \partial V^n_g=\emptyset$. By condition on our choice of $f_0:U_0\to V_0$, $CV_{f_0}$ lies in a single component $R$ of $V_{f_0}^n\cap Q$. Thus $CV_g\subset R_g:=H_{g,n}(R)$. We define $H_{g,n+1}$ on the annulus $P\setminus  f^{-1}(R)$ by lifting the map $H_{g,n}: Q\setminus R\to Q_g\setminus R_g$ whose natural extensions to $\partial P$ coincide with $H_{g,n}|_{\partial P}$. Finally, extend the definition of $H_{g,n+1}$ to the single component $P\cap f^{-1}(R)$ as a qc map. 
\end{itemize}

Note that $H_{g,n}=H_{g,n+1}$ outside $V^n_{f_0}$. Outside $\overline{V^n_{f_0}}$, the maximal dilatation of $H_{g,n}$ is bounded from above by that of $h_g$. It follows that $H_{g,n}$ converges locally uniformly to a qc map $H_g$ from  $\C\setminus K_*(f_0)$ onto $\C\setminus K_*(g)$, for each $g\in E$.
\end{proof}

\begin{lemma}\label{lem:connectedJuliasets}
Let $f_j: U_j\to V_j$, $j=1,2$ be GPLs such that
\begin{enumerate}
\item[(1)] $f_j$ has no critical point in $U_j\setminus K(f_j)$;
\item[(2)] for any component $Q$ of $V_j$, $K(f_j)\cap Q$ is connected. 
\end{enumerate}
Let $H:\C\to\C$ be a qc map such that $H\circ f_1=f_2\circ H$ on $U_1$ and let $\varphi:\C\to \C$ be a qc map which provides a hybrid conjugacy from $f_1$ to $f_2$ such that 
\begin{enumerate}
\item[(3)] for each connected component $J$ of $K(f_1)$, $\varphi(J)$ and $H(J)$ lies in the same component of $V_2$. 
\end{enumerate}
Then there is a qc map $\psi:\C\to \C$ such that 
\begin{itemize}
\item $\psi\circ f_1=f_2\circ \psi$ on $U_1$;
\item $\psi=H$ outside $V_1$;
\item $\psi=\varphi$ on $K(f_1)$.
\end{itemize}  
\end{lemma}
\begin{proof} Let $P_1\Subset U_1$ be a neighborhood of $K(f_1)$ such that 
\begin{itemize}
\item $\varphi$ is well-defined on $P_1$ with $P_2=\varphi(P_1)\Subset U_2$;
\item $\varphi\circ f_1=f_2\circ \varphi$ on $P_1$;
\item for each component $Q$ of $V_1$, $P_1\cap Q$ is a Jordan disk. 
\end{itemize}
For $j=1,2$, let $P_j^0=P_j$ and for each $n\ge 0$, let $P_j^{n+1}=(f_j|_{U_j})^{-1}(P_j^n)$. Take $\psi_0$ to be an arbitrary qc map such that 
\begin{itemize}
\item $\psi_0=H$ on $\C\setminus V_1$;
\item $\psi_0=\varphi$ on $P_1^0$. 
\end{itemize}
Note that our assumptions imply that for each component $Q$ of $U_j$, $f_j|_{Q\setminus P_j^1}$ is an un-branched covering onto $Q'\setminus P_j^0$ for some component $Q'$ of $V_j$. Thus there is a qc map $\psi_1:\C\to \C$ such that 
\begin{itemize}
\item $\psi_1=H$ on $\C\setminus V_1$;
\item $f_2\circ \psi_1=\psi_0\circ f_1$ on $U_1$;
\item $\psi_1=\varphi$ on $P_1^1$. 
\end{itemize}
Since $\psi_1=\psi_2$ outside $V$ and also in a neighborhood of $K(f_1)$ (coinciding with $\varphi$),  $\psi_0$ is isotopic to $\psi_1$ rel $(\C\setminus V_1)\cup K(f_1)$.
Since $K(f_j)$ contains all critical values of $f_j|_{U_j}$, by successively lifting the isopoty, we obtain a sequence $\psi_n:\C\to \C$ of qc maps, $n=0,1,\ldots$, such that 
\begin{itemize}
\item[(i)] $\psi_n=H$ on $\C\setminus V_1$;
\item[(ii)] $\psi_n=\varphi$ on $K(f_1)$;
\item[(iii)] $f_2\circ\psi_{n+1}=\psi_n\circ f_1$ on $U_1$. 
\end{itemize}
Since $\psi_n$ has uniformly bounded maximal dilatation and $\psi_n(z)$ eventually stablizes for each $z\in\C$, $\psi_n$ converges locally uniformly to a qc map $\psi$ which satisfies all the requirements of the lemma.  
\end{proof}

\begin{proof}[Proof of Proposition~\ref{prop:connected set}] Since $\mathcal{V}_W$ is also an equipped neighborhood of $f$, we may assume without loss of generality that $f=f_0$. Write $E=\mathcal{H}_*(f_0,\mathcal{V}_W)$.  
For each $g\in E$, let $H_g:\C\setminus K_*(f_0)\to \C\setminus K_*(g)$ be given by Lemma~\ref{lem:hgextension}, which is a qc map and satisfies $H_g\circ f_0=g\circ H_g$ on $U_{f_0}\setminus K_*(f_0)$.

We shall use~\cite[Theorem 1.3]{PT} which is a rigidity result based on~\cite{KSS} and~\cite{KL}.   
A component $J_g$ of $K_*(g)$ is called {\em critical} (resp. {\em trivial}) if it contains a critical point of $g$ (resp. it is a singleton). 

If $K_*(f_0)$ has no (pre)periodic critical component, then so does $g\in E$. In this case, Theorem 1.3 in~\cite{PT} immediately implies that $H_g$ extends to a hybrid equivalence between $f$ and $g$. Hence $E\subset \mathcal{H}_\#(f_0, \mathcal{V}_W)$.

We shall now deal with the case that $K_*(f_0)$ has a periodic critical component.
Let us call a connected component $Q$ of $(g|_{U_g})^{-n}(V)$ a $g$-puzzle piece of depth $n$. For $g\in E$, the map $H_g$ determines a bijection $\alpha_g$ between $f_0$-puzzles pieces and $g$-puzzle pieces for each fixed depth $n$: each $f$-puzzle piece $Q$ of depth $n$ corresponds to a unique $g$-puzzle piece $Q_g$ of the same depth with the property $\partial Q_g=H_g(\partial Q)$. Moreover, $f_0|_{\partial Q}$ and $g|_{\partial Q_g}$ have the same degree.  For a disjoint union $T=\bigcup_i P_i$ of puzzle pieces $P_i$ of $f_0$, we write $\alpha_g(T)=\bigcup_i \alpha_g(P_i)$.

The map $H_g$ also determines a bijection, still denoted by $\alpha_g$, between components of $K_*(f_0)$ and components of $K_*(g)$: For each component $J$ of $K_*(f_0)$, letting $Q_n$ denote the unique $f_0$-puzzle piece of depth $n$ which contains $J$, then $\alpha_g(J):=J_g:=\bigcap_{n=0}^\infty Q_n^g\in \Omega_g$ is a component of $K_*(g)$. Note that 
\begin{itemize}
\item $\alpha_g(f_0(J))=g(\alpha_g(J))$;
\item if $J$ contains a critical point of $f_0$, then $J_g$ contains a critical point of $g$. 
\end{itemize}
By \cite{QY,KvS}, a critical component $J_g$ of $K_*(g)$ is non-trivial if and only if $J_g$ is preperiodic. Thus $\alpha_g$ induces a bijection between $\Omega_{f_0,0}$ and $\Omega_{g,0}$, where $\Omega_{g,0}$ denotes the collection of non-trivial critical components of $K_*(g)$.

Let $\Omega_g$ denote the collection of non-trivial components of $K_*(g)$ which intersect the critical orbits of $g$ and let $L(g)$ denote the union of these components.  Since each element of $\Omega_g$ 
is a $g$-iterated image of some element of $\Omega_{g,0}$, $\alpha_g$ induces a bijection between $\Omega_{f_0}$ and $\Omega_g$. Let $k$ be a large positive integer such that 
\begin{itemize}
\item each puzzle piece of $f_0$ of depth $k$ contains at most one element of $\Omega_{f_0}$, and does not contain any critical point of $f_0$ outside $L(f_0)$.
\end{itemize}
Then for $g\in E$, each puzzle piece of $g$ of depth $k$ contains at most one element of $\Omega_g$ and does not contain any critical point of $g$ outside $L(g)$. Let $V'_g$ (resp. $U'_g$) denote the union of $g$-puzzle pieces of depth $k$ (resp. $k+1$) that intersect $L(g)$. Then $U_g'=\alpha_g(U_{f_0}')$ and $V_g'=\alpha_g(V_{f_0}')=V_{f_0}'$. When $\Omega_g\not=\emptyset$, the map $g: U_g'\to V_g'$ is a GPL whose filled Julia set is equal to $L(g)$. 
  
For each $g\in E$, let $\varphi_g: \C\to \C$ be a qc map which provides a hybrid conjugacy between $f_0: U_{f_0}\to V_{f_0}$ and $g: U_g\to V_g$. The dependence of $\varphi_g$ in $g$ is unclear.

Let $E_0$ be any subset of $E$ with infinitely many elements. We shall use the pigeon-hole principle to show that there exists $g_1, g_2\in E_0$ with $g_1\not=g_2$ such that $H_{g_2}\circ H_{g_1}^{-1}$ and $\varphi_{g_2}\circ \varphi_{g_1}^{-1}|_{K_*(g_1)}$ can be glued together to form a hybrid equivalence between $g_1$ and $g_2$, so that $g_2\in \mathcal{H}_\#(g_1, \mathcal{V}_W)$.

Indeed, the map $\varphi_g$ send an element of $\Omega_{f_0,0}$ (resp. $\Omega_{f_0}$) onto an element of $\Omega_{g,0}$ (resp. $\Omega_g$). Let $\beta_g: \Omega_{f_0}\to \Omega_g$ denote the bijection induced by $\varphi_g$. 
Since $\beta_g\circ \alpha_g^{-1}$ is a permutation of the finite set $\Omega_{f_0}$, there exists an infinite subset $E_1$ of $E_0$ such that for any $g_1, g_2\in E_1$, $\beta_{g_2}\circ \alpha_{g_2}^{-1}=\beta_{g_1}\circ \alpha_{g_1}^{-1}$.  
It follows that for any component $U'$ of $U_{f_0}'$, $\varphi_{g_2}\circ \varphi_{g_1}^{-1}$ maps $L(g_1)\cap \alpha_{g_1}(U')$ onto $L(g_2)\cap \alpha_{g_2}(U')$. 
Fix $g_1\in E_1$. Then, by Lemma~\ref{lem:connectedJuliasets}, for any $g\in E_1$,  
there is qc map $\psi_g: \C\to \C$ such that 
\begin{itemize}
\item $\psi_g|_{\C\setminus V'_{g_1}}=H_{g}\circ H_{g_1}^{-1}$;
\item $ \psi_g=\varphi_g\circ \varphi_{g_1}^{-1}$ on $L(g_1)$;
\item $g\circ \psi_g=\psi_g\circ g_1$ on $U_{g_1}'$. 
\end{itemize}

Let us consider the maps $h_{g}^{-1}\circ \psi_g\circ h_{g_1}|_{\partial U_{f_0}'}$, $g\in E_1$. They are lifts of the identity map on $\partial V_{f_0}$ by the map $f_0$, hence they lie in a finite set. Therefore, there is an infinite subset $E_2$ of $E_1$ such that $h_g^{-1}\circ \psi_g\circ h_{g_1}|_{\partial U'_{f_0}}$ are the same for all $g\in E_2$. 

Fix $g_2\in E_2$. Then for any $g\in E_2$, $h_g^{-1}\circ \psi_g=h_{g_2}^{-1}\circ \psi_{g_2}$ on $\partial U_{f_0}'$. That is, $\psi_g\circ \psi_{g_2}^{-1}=h_g\circ h_{g_2}^{-1}$ on $\partial U_{g_2}'$.  

It follows that $H_g\circ H_{g_2}^{-1}=\psi_{g}\circ \psi_{g_2}^{-1}$ on $(g_2|_{U'_{g_2}})^{-n}(\partial U'_{g_2})$ for any $n\ge 0$ and any $g\in E_2$.
In particular, for any puzzle piece $Q$ of $f_0$ which intersect $L(g_2)$, $\psi_g\circ \psi_{g_2}^{-1}:Q_{g_2}\to Q_g$ is a qc extension of $H_g\circ H_{g_2}^{-1}|_{\partial Q_{g_2}}$. Then maximal dilatation of the extension is bounded from above by a constant independent of $Q$. By~\cite[Theorem 1.3]{PT}, it follows that $H_g\circ H_{g_2}^{-1}$ extends to a hybrid conjugacy from $g$ to $g_2$. 

Thus $E$ intersects only finitely many equivalence classes $\mathcal{H}_\#(g, W)$, $g\in \mathcal{V}_W$. By Lemma~\ref{lem:closedH}, the intersection of $E$ with each of these equivalence classes is closed in $E$. Since $E\ni f_0$ and $E$ is connected, $E$ is contained in the equivalence class $\mathcal{H}_{\#} (f_0, W)$.  
\end{proof}

\section{Lyubich's horizontal-vertical decomposition}\label{sec:lyubichdecomp}
In this section, we study Lyubich's horizontal-vertical decomposition of $\mathcal{B}_W$ in the GPL setting. The argument here largely follows Lyubich's original argument, with modifications needed for our setting. These results will be used in \S~\ref{sec:hybridclass} (Lemma~\ref{lem:mating0}) to study the tangent map of the `mating' maps.

Throughout this section, $f_0\in \mathcal{Q}_{W_0}$ for some open subset $W_0$ of $\C$ and let $W$ be an open subset of $\C$ with 
$K_*(f_0)\subset W\subset W_0.$
Let $E^h(f_0, W)$ denote the collection of $v\in \mathcal{B}_W$ for which there is a continuous function $\beta: \mathbb{C} \to \mathbb{C}$ with the following properties: 
\begin{itemize}
\item For $z$ in a neighborhood of $K_*(f_0)$ in $\C$, we have
\begin{equation}\label{eqn:horizontal}
v(z)=\beta(f_0(z))-f'_0(z)\beta(z);
\end{equation}
\item $\beta(z)/dz$ is a qc vector field, i.e. $\beta$ has locally integrable distributional derivatives $\partial\beta$ and $\bar{\partial}\beta$, and $\bar{\partial} \beta$ is essentially bounded;
\item $\overline{\partial} \beta(z) =0$ holds for a.e. $z\in K_*(f_0)$.  
\end{itemize}

Let $E^{vert}_*(f_0,W)$ denote the space consisting of $v\in \mathcal{B}_W$ for which there is a holomorphic function $\gamma:\mathbb{C} \setminus K_*(f_0)\to \mathbb{C}$ such that 
\begin{equation}\label{eqn:vertical}
v(z)=\gamma(f_0(z))- f'_0(z)\gamma(z)
\end{equation}
holds for all $z\in W\setminus K_*(f_0)$ and such that
\begin{equation}\label{eqn:gamma0} 
\gamma(z)\to 0 
\mbox{ as }z\to\infty. 
\end{equation}

For $v\in E^h(f_0,W)$ (resp. $v\in E^{vert}_*(f_0, W)$), the vector field $v(z)/dz$ is called a {\em horizontal} (resp. {\em (normalized) vertical}) vector field for $j_{W_0}f_0$ (or $f_0$ if no confusion arises).

\begin{theorem}\label{thm:codimEh0} For each $f_0\in \mathcal{Q}_{W_0}$, there is a nonnegative integer $D=D(f_0,W_0)$ such that for any open subset $W$ of $\C$ with $K_*(f_0)\subset W\subset W_0$, $E^h(f_0,W)$ is a closed subspace of $\mathcal{B}_{W}$ of codimension $D$. Moreover, if in addition $W\Subset W_0$, then $\text{dim}(E^{vert}_*(f_0,W))=D$ and 
\begin{equation}\label{eqn:decomp}
\mathcal{B}_{W}=E^h(f_0, W)\oplus E^{vert}_*(f_0,W)
\end{equation}
holds as Banach spaces.  
\end{theorem}

We shall show in \S\ref{sec:codim} that $D(f_0, W_0)$ coincides with $d(j_{W_0}f_0)$ defined in (\ref{eqn:criticality}).
Let us note that the assumption $W\Subset W_0$ is necessary for the decomposition (\ref{eqn:decomp}) to hold; see \S\ref{subsec:counterex}.

Lyubich's technical lemma gives uniform estimates of vectors in $E^{vert}_*(g,W)$ for $g$ in a small neighborhood of $f_0$ in $\mathcal{Q}_W$. We use this result to study the continuity of $D(g,W)$ and the vertical spaces $E^{vert}_*(g,W)$.

\begin{theorem}\label{thm:continuity} Suppose that $f_n\to f_0$ in $\mathcal{Q}_{W_0}$, $K_*(f_0)\subset W\Subset W_0$, and $d(j_{W_0}f_n)=d(j_{W_0}f_0)$ for all $n$. Then $$\text{dim}(E^{vert}_*(f_n,W))=\text{dim}(E^{vert}_*(f_0,W))$$ for all sufficiently large $n$. Moreover, 
$E^{vert}_*(f_n,W)$ converges to $E^{vert}_*(f_0,W)$.
\end{theorem}

In general, we do not expect the horizontal spaces $E^h(f,W)$ to depend on $f$ continuously, even within the criticality stratum $\{f\in \mathcal{Q}_{W_0}: d(j_{W_0}f)=d(j_{W_0}f_0)\}$. Nevertheless, our result shows that for maps $g$ in the critical stratum, locally, the horizontal spaces $E^h(f,W)$ have definite `angle' with the vertical spaces $E^{vert}_*(f,W)$.  

We shall also prove the following results which will be used later.

\begin{corollary}\label{cor:Dfgerm} For any GPLs $f:U\to V$ and $\widetilde{f}:\widetilde{U}\to \widetilde{V}$ which are qc equivalent, we have
$\text{dim}(E_*^{vert}(f,U))=\text{dim}(E^{vert}_*(\widetilde{f},\widetilde{U}))$.
\end{corollary}

\begin{lemma}\label{lem:dense} Let $f_0\in \mathcal{Q}_{W_0}$ and let $W$ and $\widetilde{W}$ be open subsets of $\C$ such that 
$W_0\Supset W\Supset \widetilde{W}\supset K_*(f_0)$. Assume $\widetilde{W}$ is a finite union of Jordan disks with pairwise disjoint closure. Then $\{v|_{\widetilde{W}}: v\in E^h(f_0, W)\}$ is dense in $E^h(f_0|_{\widetilde{W}},\widetilde{W}).$
\end{lemma}

\subsection{The decomposition and technical lemma}

In the following, we assume that $f_0, W, W_0$ are as in Theorem~\ref{thm:codimEh0}. 
\begin{lemma}\label{lem:verticalgamma}
Each $v\in E^{vert}_*(f_0,W)$ has a unique analytic extension to a holomorphic function defined on $W_0$. 
\end{lemma}

\begin{proof} 
We first prove the existence of a holomorphic extension $\widehat{v}: W_0\to \C$.  By definition of a vertical vector field, there exists a holomorphic function $\gamma:\C\setminus K_*(f_0)\to \C$ 
such that (\ref{eqn:vertical}) holds for $z\in W\setminus K_*(f_0)$ and such that (\ref{eqn:gamma0}) holds. 
Since $f_0\in 
\mathcal{B}_{W_0}$, for $z\in W_0\setminus K_*(f_0)$, $f_0(z)\not\in K_*(f_0)$. 
Thus
$$\widehat{v}(z)=\left\{\begin{array}{ll}
\gamma(f_0(z))-f'_0(z)\gamma(z)& \mbox{ if } z\in W_0\setminus K_*(f_0);\\
v(z) &\mbox{ if } z\in K_*(f_0)
\end{array}
\right.
$$
is holomorphic in $W_0$. Clearly, $\widehat{v}|_{W}=v$.

Now let us prove the uniqueness. It suffices to show that if $\widehat{v}$ is holomorphic in $W_0$ and $\widehat{v}|_{W}=0$, then $\widehat{v}\equiv 0.$
Let $A_0$ be an arbitrary connected component of $W_0$. Since $f_0\in \mathcal{Q}_{W_0}$, $A_0$ contains a point $a\in K_*(f_0)$. Since $W\supset K_*(f_0)$, $W\ni a$. Let $A$ be the component of $W$ which contains $a$. Then $A\subset A_0$. Since $\widehat{v}|_{A}=0$, we have $\widehat{v}|_{A_0}=0$. 
\end{proof}

\begin{lemma}\label{lem:expansion} Let $f:U\to V$ be a GPL such that $f$ has no critical point in $U\setminus K(f)$. Let $V_k=(f|_U)^{-k} (V)$, $k=0,1,\dots$. Then 
$$\min_{z\in \partial V_k} |(f^k)'(z)|\to\infty, \text{ as } k\to\infty.$$
\end{lemma}
\begin{proof} Let $\rho(z)|dz|$ (resp. $\rho_0(z)|dz|$) denote the hyperbolic metric on $V\setminus K(f)$ (resp. $U\setminus K(f)$). Since $f$ has no critical point in $U\setminus K(f)$, $f:U\setminus K(f)\to V\setminus K(f)$ is a covering, hence 
$|f'(z)|\rho(f(z))=\rho_0(z)$ holds for all $z\in U\setminus K(f)$. By Schwarz Lemma, 
$\rho_0(z)>\rho(z)$ for all $z\in U\setminus K(f)$. Thus for all $z\in U\setminus K(f)$, 
$$\|f'(z)\|=\rho (f(z))|f'(z)|/\rho(z)>1.$$
Thus for all $k\ge 1$ and $z\in V_k$, $\|(f^{k-1})'(z)\|=\rho(f^{k-1}(z))|(f^{k-1})'(z)|/\rho(z)>1.$ As $k\to\infty$, $d(\partial V_k, K(f))\to 0$, so $\min_{z\in \partial V_k}\rho(z)\to\infty$. Since $\max_{w\in \partial V_1}\rho(w)<\infty$, 
it follows that $\min_{z\in \partial V_k}|(f^{k-1})'(z)|\to\infty$ as $k\to\infty$. Since $\min_{w\in\partial V_1}|f'(w)|>0$, the statement follows. 
\end{proof}
\begin{lemma}[Lyubich's technical lemma]\label{lem:Lyutech}  Consider $f_0\in \mathcal{Q}_{W_0}$. Let $X, Y$ be open subsets of $\C$ with $K_*(f_0)\subset X\Subset Y\Subset W_0$. 
Then there is a neighborhood $\mathcal{V}_{W_0}(f_0)$ of $f_0$ in $\mathcal{Q}_{W_0}$ and a constant $C>0$ such that for any $f\in \mathcal{V}_{W_0}(f_0)$ and $v\in E^{vert}_*(f, Y)$ with $\|v\|_X=1$, we have $X\supset K_*(f)$ and 
$$\|v\|_{Y}\le C, \sup_{z\in \C\setminus X}|\gamma(z)|\le C,$$
where $\gamma:\mathbb{C}\setminus K_*(f)\to \mathbb{C}$ is a holomorphic function such that (\ref{eqn:vertical}) and (\ref{eqn:gamma0}) hold (with $f_0$ replaced by $f$).     
\end{lemma}

\begin{proof} Shrinking $X$ if necessary, we may assume that $j_{W_0}f_0$ has a GPL representative $f_0: U\to V$ with $V\Subset X$ and $f_0$ has no critical point in $\overline{U}\setminus K_*(f_0)$. Let $V_0=V$ and define inductively $V_{k+1}=f_0^{-1}(V_k)\cap U$ for each $k\ge 0$. Then $\bigcap_{k=0}^\infty V_k=K_*(f_0)$. By Lemma~\ref{lem:expansion}, there is a positive integer $k$ such that $|(f_0^k)'(z)|> 2$ holds for all $z\in \partial V_k$. 
Fix $k$ and take a small neighborhood $\mathcal{V}_{W_0}(f_0)$ of $f_0$ in $\mathcal{B}_{W_0}$.
Then, as in Lemma~\ref{lem:persistglp}, there is a holomorphic motion $h_g$ of $\mathbb{C}$ over $(\mathcal{V}_{W_0}(f_0), f_0)$ such that for each $f\in \mathcal{V}_{W_0}$,
\begin{itemize}
\item $h_f(V)=V$; 
\item $f: h_f(U)\to V$ is a GPL representative $j_{W_0}(f)$;     
\item for $z\in \partial U$, $h_f(f_0(z))=f_0(h_f(z))$. 
\end{itemize}
By continuity, for $f\in \mathcal{V}_{W_0}(f_0)$, $|(f^k)'(z)|>2$ holds for all $z\in \partial V_k^f$, where $V_k^f=(f|_{U_f})^{-k}(V)$.

Now fix $f\in \mathcal{V}_{W_0}$, let $v\in E^{vert}_*(f, Y)$ with $\|v\|_{X}=1$ and let $\gamma:\mathbb{C}\setminus K_*(f)\to \mathbb{C}$ be a holomorphic function such that (\ref{eqn:vertical}) and (\ref{eqn:gamma0}) holds (with $f_0$ replaced by $f$). 
For $z\in \partial V^k_f$, 
$$v(f^{k-1}(z))+f'(f^{k-1}(z))v(f^{k-2}(z))+\dots+ f'(f(z)) v(z)  =\gamma(f^k(z))-(f^k)'(z) \gamma(z).$$
The left hand side is bounded from above by a constant $C_1$ (depending only on  $k$). Thus 
$$2|\gamma(z)|-|\gamma(f^k(z))|\le C_1.$$
Now let $M=\sup_{z\in \mathbb{C}\setminus V_f^k} |\gamma(z)|$. 
By the maximum principle, $M=|\gamma(z_0)|$ for some $z_0\in\partial V^k_f$. 
Thus $$M\le 2|\gamma(z_0)|-|\gamma(f^k(z_0)|\le C_1.$$ 
In particular, $$\sup_{z\in \mathbb{C}\setminus X} |\gamma(z)|\le M\le C_1.$$
Shrinking $X$ further, we may assume $f(\partial W_0)\cap X=\emptyset$. Enlarging $Y$ so that $\partial Y$ is very close to $\partial W_0$. Then for $z\in \partial Y$, $|\gamma(z)|\le M, |\gamma(f(z))|\le M$, so
$$|v(z)|\le |\gamma(f(z))|+|f'(z)||\gamma(z)|$$
is bounded from above by a constant $C_2$. By the maximum principle, it follows that $\|v\|_Y\le C_2$. 
\end{proof}

 

\begin{lemma}\label{lem:vertcomp} Let $f_n\to f_0$ in $\mathcal{Q}_{W_0}$ and $K_*(f_0)\subset W\Subset W_0$. Consider $v_n\in E^{vert}_*(f_n,W)$ with $\|v_n\|_{W}\le 1$. Then there is a subsequence $v_{n_k}$ which converges in $\mathcal{B}_{W}$ to some $v\in E^{vert}_*(f,W)$.
\end{lemma}
\begin{proof} For each $n$, let $\gamma_n:\mathbb{C}\setminus K_*(f_n)\to \mathbb{C}$ be the holomorphic function with $\gamma_n(z)\to 0$ as $z\to\infty$ and such that 
$$v_n(z)=\gamma_n(f_n(z))-f_n'(z)\gamma(z)$$
holds for $z\in W\setminus K_*(f_n)$. By Lemma~\ref{lem:Lyutech}, $\gamma_n$ is locally uniformly bounded in $\widehat{\C}\setminus K_*(f_0)$. By Montel's theorem, there is a subsequence $n_k$ such that $v_{n_k}$ converges locally uniformly in $W$ to a holomorphic function $v:W\to \C$ and $\gamma_{n_k}$ converges locally uniformly in $\C\setminus K_*(f_0)$ to a holomorhic function $\gamma:\C\setminus K_*(f)\to \C$ with $\gamma(z)\to 0$ as $z\to\infty$. Moreover, 
$$v(z)=\gamma(f_0(z))-f'_0(z)\gamma(z)$$
holds for $z\in W\setminus K_*(f_0)$. Since $f_0\in \mathcal{B}_{W}$, $v$ extends continuously to $\overline{W}$, so $v\in E^{vert}_*(f_0,W)$.
Since $\gamma_{n_k}(f_{n_k}(z))\to \gamma(f_0(z))$ and $f_{n_k}'(z)\to f_0'(z)$ uniformly in $z\in \partial W$, it follows that 
$$\|v_{n_k}-v\|_{\partial W}\to 0,$$
which implies by the maximum modulus principle that $\|v_{n_k}-v\|_{W}\to 0$.
\end{proof}

The following is a straightforward generalization of \cite[Lemma 4.7]{L2}.
\begin{lemma} [Lyubich's decomposition lemma]\label{lem:Lyudecomp} 
Assume that $W\Subset W_0$. Then the decomposition (\ref{eqn:decomp}) holds as vector spaces. 
\end{lemma}

\begin{proof} The fact that $E^h(f_0,W)\cap E_*^{vert}(f_0,W)=\{0\}$ follows from the same argument as in the proof of \cite[Lemma 4.7]{L2}. Indeed, for $v\in E^h(f_0,W)\cap E^{vert}_*(f_0,W)$, let $\beta$ and $\gamma$ be such that (\ref{eqn:horizontal}), (\ref{eqn:vertical}) and (\ref{eqn:gamma0}) holds. Let $\delta=\beta-\gamma$. Then $\delta(z)/dz$ is a qc vector field, and $\delta(f_0(z))-f_0'(z) \delta(z)=0$ holds for $z\in W'\setminus K_*(f_0)$ for some open set $W'$ with $K_*(f_0)\subset W'\subset W$. Let $f_0:U\to V$ be a GPL representative of $j_{W_0}f_0$ with $U\subset W'$ and such that $f_0$ has no critical point in $U\setminus K_*(f_0)$. Let $V_k=(f_0|_U)^{-k}(V)$ for each $k\ge 0$. For $z\in U\setminus K_*(f_0)$, let $n(z)$ denote the minimal positive integer such that $f_0^{n(z)}(z)\not\in V\setminus U$. Then $$\delta(z)=\delta(f_0^{n(z)}(z))/(f_0^{n(z)})'(z).$$
As $z\to K_*(f_0)$, $n(z)\to \infty$, so by Lemma~\ref{lem:expansion}, $|(f_0^{n(z)})'(z)|\to\infty$, hence $\delta(z)\to 0$. This shows that $\delta$ extends to a continuous function $\widetilde{\delta}$ on $\C$ with $\delta|_{K_*(f_0)}=0$. Thus $\gamma=\beta-\delta$ 
extends to a continuous map $\widetilde{\gamma}=\beta-\widetilde{\delta}$ on $\C$ such that $\widetilde{\gamma}(z)/dz$ is a qc vector field on $\C$ and $\bar{\partial}\widetilde{\gamma}=\bar{\partial} \beta=0$ a.e. on $K_*(f_0)$. Thus $\widetilde{\gamma}$ is holomorphic in $\C$ and hence it is identically zero since $\gamma(z)\to 0$ as $z\to\infty$. Consequently, $v=0$ by (\ref{eqn:vertical}).

Let us show that any $v\in \mathcal{B}_{W}$ can be written as in the form $v=u+w$ with $u\in E^h(f_0,W)$ and $w\in E^{vert}_*(f_0,W)$. 

{\bf Step 1.} Choose a GPL representative $f_0:U\to V$ of $j_{W_0} f_0$ such that $f_0$ has no critical point in $\overline{U}\setminus K_*(f_0)$ and $U\Subset W$.

{\bf Step 2.} Find a qc vector field $\alpha(z)/dz$ in $\C\setminus U$ such that
for $z \in \partial U$,
\begin{equation}\label{eqn:alpharec} 
\alpha(z)=(\alpha(f_0(z))-v(z))/f'_0(z),
\end{equation}
such that $\bar{\partial}\alpha$ is compactly supported and $\alpha(z)\to 0$ as $z\to\infty$. 

For example, we can choose $\alpha$ such that $\alpha(z)=0$ for $z\in \C\setminus V$, use (\ref{eqn:alpharec}) to define $\alpha|_{\partial U}$ and then make a smooth interpolation. 

{\bf Step 3.} Extend $\alpha(z)/dz$ to a qc vector field in $\C\setminus K_*(f_0)$ such that 
(\ref{eqn:alpharec}) holds for $z\in U\setminus K_*(f_0)$.

Indeed, $K_*(f_0)=\bigcap_{n=1}^\infty V_n$, where $V_0=V$ and $V_{n+1}=(f_0|_U)^{-1}(V_n)$ for each $n\ge 0$. Since $f_0$ has no critical point in $\overline{U}\setminus K_*(f_0)$, we can inductively define $\alpha$ inductively in $V_n\setminus V_{n+1}$ using the formula (\ref{eqn:alpharec}). 

Note that 
$$\|\bar{\partial }\alpha\|_{L^\infty (\C\setminus K_*(f_0))}=\|\bar{\partial}\alpha\|_{L^\infty(\C\setminus U)}<\infty.$$

{\bf Step 4.} Find $\widetilde{u}\in E^h(f_0,U)$ and $\widetilde{w}\in E^{vert}_*(f_0,U)$ such that $v=\widetilde{u}+\widetilde{w}$.

To this end, let $\mu=\bar{\partial} \alpha$ on $\C\setminus K_*(f_0)$ and extend it to an $L^\infty$ function on $\mathbb{C}$ by setting $\mu(z)=0$ for $z\in K_*(f_0)$. Solving the Beltrami equation 
$$\bar{\partial} \beta=\mu,$$
we obtain a qc vector field $\beta(z)/dz$ on $\mathbb{C}$ with $\beta(z)\to 0$ as $z\to\infty$. 
Let $$\widetilde{u}(z)=\beta(f_0(z))-f'_0(z) \beta(z)$$ which is continuous in $\overline{U}$. As 
$\bar{\partial} \widetilde{u}(z)=0$ holds a.e. on $U$, by Weyl's lemma, we see that $u\in \mathcal{B}_{U}$.

Let $\gamma=\alpha-\beta$. Then $\bar{\partial}\gamma(z)=0$ a.e. on $\mathbb{C}\setminus K_*(f_0)$, so $\gamma$ is holomorphic in $\mathbb{C}\setminus K_*(f_0)$. Clearly, $\gamma(z)\to 0$ as $z\to \infty$.  Moreover, 
$$\widetilde{w}(z):=v(z)-u(z)\in E^{vert}_*(f_0, U),$$ since  
$$\gamma(f_0(z))-f'_0(z)\gamma(z)=\widetilde{w}(z)$$
holds for $z\in U\setminus K_*(f_0)$. 

{\bf Step 5.} By Lemma~\ref{lem:verticalgamma}, $\widetilde{w}$ extends to a holomorphic function $w\in E^{vert}_*(f_0, W)$. Thus $u:=v-w\in \mathcal{B}_{W}$. Since $u=\widetilde{u}$ on $U$,  $u\in E^h(f_0,W)$. 
\end{proof}
For $f_0\in \mathcal{Q}_{W_0}$ and $K_*(f_0)\subset W\Subset W_0$, let $P_{f_0}^h:\mathcal{B}_{W}\to E^h(f_0,W)$ and $P_{f_0}^{vert}:\mathcal{B}_{W}\to E^{vert}_*(f_0,W)$ denote the projections associated to the direct sum decomposition given by the lemma above.  

\begin{proof}[Proof of Theorem~\ref{thm:codimEh0}] Assume first that $W\Subset W_0$. By Lemma~\ref{lem:vertcomp} (applying to $f_n=f_0$), the unit ball of $E^{vert}_*(f_0,W)$ is compact. Thus $$D(f_0, W):=\text{dim}(E^{vert}_*(f_0,W))<\infty.$$

By Lemma~\ref{lem:verticalgamma}, each $v\in E^{vert}_*(f_0,W)$ extends uniquely to a holomorphic function on $W_0$. Thus $D(f_0,W)$ does not depend on $W$. That is, if $W'$ is another open subset of $\C$ with $K_*(f_0)\subset W'\Subset W_0$, then $D(f_0,W)=D(f_0,W')$. We denote the common number by $D=D(f_0,W_0)$. 

In order to show that $E^h(f_0,W)$ is closed and (\ref{eqn:decomp}) holds as Banach spaces, it suffices to show that the projection $P^{vert}_{f_0}:\mathcal{B}_{W}\to E^{vert}_*(f_0,W)$ is continuous. To see this, we review the construction in  
the proof of Lemma~\ref{lem:Lyudecomp}. In step 1, we choose the GPL representative $f_0:U\to V$ such that $U\Subset W$. In Step 2, we can choose $\alpha$ such tat  $\|\bar{\partial}\alpha\|_{L^\infty(\C\setminus U)}$ is bounded from above by $C\|v\|_U\le C\|v\|_{W}$, where $C$ is a constant. The same upper bound holds for the extension of $\alpha$ to $\C\setminus K_*(f_0)$ in Step 3. In Step 4, we have  $\|\bar{\partial}\beta\|_\infty\le C\|v\|_{W}$ and hence $\|\widetilde{w}\|_{U}\le C'\|v\|_{W}$. 
For $w$ constructed in Step 5, applying Lemma~\ref{lem:Lyutech} (to $X=U$, $Y=W$ and $f=f_0$), we have $\|w\|_{W}=O(\|\widetilde{w}\|_{U})=O(\|v\|_{W})$. As $w=P^{vert}_{f_0}(v)$, this proves that $P^{vert}_{f_0}$ is continuous.

It remains to show that $E^h(f,W_0)$ is closed in $\mathcal{B}_{W_0}$ with codimension $D$. Take an open set $W$ with $K_*(f)\subset W\Subset W_0$ such that $W$ is a finite union of Jordan disks with pairwise disjoint closure. By Mergelyan's theorem, the natural inclusion $\iota: \mathcal{B}_{W_0}\to \mathcal{B}_{W}$ has dense image. Let $\psi=P^{vert}_{f_0}\circ \iota$. Then $\psi^{-1}(0)=E^h(f_0,W_0)$. Since $\psi$ is a bounded linear surjection, it follows that the codimension of $E^h(f_0,W_0)$ in $\mathcal{B}_{W_0}$ is equal to $D(f_0,W)=D$. 
\end{proof}

\begin{proof}[Proof of Lemma~\ref{lem:dense}] Let $u\in E^h(f_0, \widetilde{W})$. By Mergelyan's Theorem, there exists a sequence $u_n$ of polynomials such that $\|u_n-u\|_{\widetilde{W}}\to 0$. Let $u_n=v_n+w_n$ with $v_n\in E^h(f_0,W)$ and $w_n\in E^{vert}_*(f_0,W)$. Then 
$$\|v_n-u+w_n\|_{\widetilde{W}}\to 0.$$
Since $v_n-u\in E^h(f_0, \widetilde{W})$ and $w_n\in E^{vert}_*(f_0,\widetilde{W})$, this implies that $\|v_n-v\|_{\widetilde{W}}\to 0$ since $P^h_{f_0}:\mathcal{B}_{\widetilde{W}}\to E^h(f_0,\widetilde{W})$ is continuous.
\end{proof}

\subsection{Continuity of the spaces of vertical vector fields}
In this subsection, we study dimension and continuity of vertical spaces.

\begin{proposition}\label{prop:lowerbd} If $f_n\to f_0$ in $\mathcal{Q}_{W_0}$ and $K_*(f_0)\subset W\Subset W_0$, then
  $$\limsup_{n\to\infty}\text{dim}(E^{vert}_*(f_n, W))\le \text{dim}(E^{vert}_* (f_0,W)).$$ 
\end{proposition}
\begin{proof} 
Let $D$ be a nonnegative integer such that $D\le \limsup_{n\to\infty} \text{dim}(E^{vert}_*(f_n,W))$. Our goal is to show that $\text{dim}(E^{ver}_*(f_0,W))\ge D$. 
Without loss of generality, we may assume that $\text{dim}(E^{vert}_*(f_n, W))\ge D$ for all $n$.
For each $n$, let us choose $u_1^n, u_2^n,\cdots, u_D^n\in E^{vert}_*(f_n, W)$ such that
$\|u_i^n\|_{W}=1$, and such that for each $2\le i\le D$, 
$$\inf_{a_1, a_2, \cdots, a_{i-1}\in \mathbb{C}}\left\|u_i^n-\sum_{j=1}^{i-1} a_j u_j^n\right\|_{W}\ge 1.$$    
By Lemma~\ref{lem:vertcomp}, after passing to a subsequence, we may assume that for each $i$, $u_i^n\to u_i$ in $E^{vert}_*(f_0,W)$. Then for each $a_1,a_2,\ldots, a_{i-1}\in \mathbb{C}$,
$$\left\|u_i-\sum_{j=1}^{i-1} a_j u_j\right\|_{W}\ge 1.$$  
Thus $u_1, u_2, \dots, u_D$ are linearly independent in $\mathcal{B}_{W}$, hence $\text{dim}(E^{vert}_*(f_0,W))\ge D$. 
\end{proof}

\begin{proposition}\label{prop:upperbd}  If $f_n\to f_0$ in $\mathcal{Q}_{W_0}$, $K_*(f_0)\subset W\Subset W_0$, and $d(j_{W_0} (f_n))=d(j_{W_0} (f_0))$ for all $n$ large, then  for any $v\in E^{vert}_*(f_0, W)$, $$P_{f_n}^{vert}(v)\to v \text{ in }\mathcal{B}_{W}.$$
In particular,
$$\liminf_{n\to\infty} \text{dim}(E^{vert}_*(f_n,W))\ge \text{dim}(E^{vert}_*(f_0,W)).$$
\end{proposition}
\begin{proof} Consider $v\in E^{vert}_*(f_0,W)$ and we review the process in the proof of Lemma~\ref{lem:Lyudecomp} to construct $P^{vert}_{f_n}v$.  Let
$\gamma:\mathbb{C} \setminus K_*(f_0)\to \mathbb{C}$ be holomorphic such that $v(z)=\gamma(f_0(z))- f'_0(z)\gamma(z)$ holds for all $z\in W\setminus K_*(f_0)$ and such that $\gamma(z)\to 0$ 
as $z\to\infty$.

{\bf Step 1.} Let $f_0: U\to V$ be a GPL representative of $j_{W_0}f_0$ with $V\Subset W$ such that $f$ has no critical point in $\overline{U}\setminus K_*(f_0)$. By Lemma~\ref{lem:persistglp}, there is a holomorphic motion $h_f$ of $\C$ over $(\mathcal{V},f_0)$ for some neighborhood $\mathcal{V}$ in $\mathcal{B}_{W_0}$ such that for all $f\in \mathcal{V}$, $h_f|_{\C\setminus V}=id$, $f:h_f(U)\to V$ is a GPL representative of $j_{W_0}f$ and such that $f\circ h_f=f_0$ holds on $\partial U$. For $n$ large enough, $f_n\in \mathcal{V}$. Let $U_n=h_{f_n}(U)$. Since $d(j_{W_0} f_n)=d(j_{W_0} f)$, $f_n$ has no critical point in $\overline{U}_n\setminus K_*(f_n)$.

{\bf Step 2.} We construct a smooth function $\alpha_n: \mathbb{C}\setminus U_n\to \mathbb{C}$ with the following properties:
\begin{itemize}
\item For each $z\in \mathbb{C}\setminus V$, $\alpha_n(z)=\gamma(z)$; 
\item For $z\in \partial U_n$, $\alpha_n(z)=(\gamma(f_n(z))-v(z))/f_n'(z)$;
\item $\|\bar{\partial} \alpha_n\|_{L^\infty(\C\setminus U_n)}\to 0$ as $n\to\infty$.
\end{itemize}


The construction of $\alpha_n|_{V\setminus U_n}$ goes as follows. We  first choose for each component $\ell$ of $\partial V$ a small annular  neighborhood $B(\ell)$. Fix $\psi_\ell\in C_0^\infty(\C)$ such that $\psi_\ell(z)=1$ for $z$ in a small neighborhood of $\ell$ and $\text{supp} (\psi_\ell)\subset B(\ell)$. 
For each component $\sigma$ of $\partial U_n$, let $A_n(\sigma)$ be the annular neighborhood of $\sigma$ which is mapped onto the corresponding annular neighborhood of $f_n(\sigma)$. For $n$ large enough, the maps $f_n|_{A_n(\sigma)}$ is a conformal map onto its image and close to $f_0$ up to the first derivative. The holomorphic function $\gamma$ is well-defined in $A_n(\sigma)$. For $z\in A_n(\sigma)$, 
\begin{align*}
Q_n(z):=& -\gamma(z)+\left(\frac{\gamma(f_n(z))-v(z)}{f_n'(z)}\right)\\
=& -\left(\frac{\gamma(f_0(z))-v(z)}{f'_0(z)}\right)+
\left(\frac{\gamma(f_n(z))-v(z)}{f_n'(z)}\right)
\end{align*}
is bounded from above by $\varepsilon_n \|v\|_{W_0}$, up to the first derivative, where $\varepsilon_n\to 0$ as $n\to\infty$. Define $q_n(z)=Q_n(z)\psi_{f_n(\sigma)} (f_n(z))$. Then $q_n(z)=0$ on $\partial A_n(\sigma)$ and $q_n(z)=Q_n(z)$ for $z\in \sigma$. We define $\alpha_n:V\setminus U_n\to \mathbb{C}$ such that
\begin{itemize}
\item $\alpha_n(z)=q_n(z)+\gamma(z)$ for $z\in A_n(\sigma)\setminus U_n$, where $\sigma$ runs over all components of $U_n$; 
\item $\alpha_n(z)=\gamma(z)$ for $z\in V\setminus\bigcup_{\sigma} A_n(\sigma)$. 
\end{itemize}

{\bf Step 3.} Since $f_n$ has no critical point in $U_n\setminus K_*(f_n)$, we may extend $\alpha_n$ to $\C\setminus K_*(f_n)$ such that 
$$v(z)=\alpha_n(f_n(z))-f_n'(z) \alpha_n(z)$$
holds for all $z\in U_n\setminus K_*(f_n)$.
The vector field $\alpha_n(z)/dz$ is qc in $\C\setminus K_*(f_n)$ and 
\begin{equation}\label{eqn:estalphan}
\|\bar{\partial}\alpha_n(z)\|_{L^\infty (\C\setminus K_*(f_n))}=\|\bar{\partial} \alpha_n(z)\|_{L^\infty (V\setminus U_n)}\le C\varepsilon_n \|v\|_{W},
\end{equation}
where $C>0$ is a constant.

{\bf Step 4.} Let $\beta_n(z)/dz$ be a qc vector field such that
\[ \bar{\partial} \beta_n(z)=
\begin{cases}
0, & z\in K_*(f_n)\\
\bar{\partial} \alpha_n(z),  & z\in \C \setminus K_*(f_n)\\
\end{cases}
\]
and $\beta_n(z)=O(1/z)$ as $z\to\infty$. Since $f_n^*(\bar{\partial} \beta_n|_V)=\bar{\partial} \beta_n|_{U_n}$,
$$\widetilde{u}_n(z):=\beta_n\circ f_n(z)-f_n'(z) \beta_n(z)$$
belongs to $E^h(f_n, U_n)$.
Moreover, 
$$\widetilde{w}_n(z):=v(z)-\widetilde{u}_n(z)\in E^{vert}_*(f_n, U_n),$$
since
\[\widetilde{w}_n(z)=\delta_n\circ f_n(z) - f_n'(z)\delta_n(z)
\]
for $z\in U_n\setminus K_*(f_n)$, where 
$\delta_n:=\alpha_n-\beta_n$ is conformal in $\mathbb{C}\setminus K_*(f_n)$ and $\delta_n(z)=O(1/z)$ as $z\to\infty$. 

{\bf Step 5.} $\widetilde{w}_n$ extends to a function $w_n$ in $E^{vert}_*(f_n,W_0)$, which is equal to $P^{vert}_{f_n}v$, i.e.,   
$$P^{vert}_{f_n}v|_{U_n}=\widetilde{w}_n.$$

Now let us show that $w_n\to v$ in $\mathcal{B}_{W}$. Note that $\|\widetilde{u}_n\|_{U_n}\to 0$, so $\|\widetilde{w}_n\|_{U_n}\le 2$ for all $n$ large. For $n$ large enough, $U_n\Supset X:=f^{-2}(V)\Supset K_*(f_0)$, so $\|w_n\|_{W}=O(\|w_n\|_{X})=O(1)$ is uniformly bounded by Lemma~\ref{lem:Lyutech} (taking $Y=W$). 
By Lemma~\ref{lem:vertcomp}, each subsequence of $w_n$ has a convergent subsequence in  $\mathcal{B}_W$. Thus it suffices to show the limit $w$ of any convergent subsequence $\{w_{n_k}\}$ of $\{w_n\}$ is equal to $v$. Indeed, for $z\in X$,
$$v(z)-w(z)=\lim_{k\to\infty} (v(z)-\widetilde{w}_{n_k}(z))=\lim_{k\to\infty} \widetilde{u}_{n_k}(z)=0.$$


We have proved the former statement of the lemma. The latter follows immediately. 
\end{proof}

\begin{proof}[Proof of Theorem~\ref{thm:continuity}]
Take an open subset $W$ of $\C$ with $K_*(f_0)\subset W\Subset W_0$. By Propositions~\ref{prop:lowerbd} and~\ref{prop:upperbd}, it follows that $\text{dim}(E^{vert}_*(f_n,W))=\text{dim}(E^{vert}_*(f_0,W))$ for all $n$ large enough. 

By Proposition~\ref{prop:upperbd}, $P^{vert}_{f_n}(v)\to v$ for each $v\in E^{vert}_*(f,W_0)$. Together with the dimension equality which has just been proved, it follows that $E^{vert}_*(f_n,W_0)$ converges to $E^{vert}_*(f,W_0)$.
\end{proof}

\begin{proof}[Proof of Corollary~\ref{cor:Dfgerm}] 
By assumption, there is 
a qc map $h: \mathbb{C}\to\mathbb{C}$ such that $h\circ f=\widetilde{f}\circ h$ holds in a neighborhood $U_0$ of near $K(f)$. We may choose $U_0$ so that $f_0: U_0\to V_0=f(U_0)$ is a GPL which is equivalent to $f:U\to V$.  Let $\mu$ be the Beltrami differential associated to $h$, 
and let $h_t$ be the solution of the Beltrami equation $\overline{\partial} h_t(z)=t\mu \partial h_t(z)$ with $h_0=id, h_1=h$. Define $f_t=h_t\circ f\circ h_t^{-1}: U_t\to V_t$, where $U_t=h_t(U), V_t=h_t(V)$.   

It suffices to show that $t\mapsto \text{dim}(E^{vert}_*(f_t, U_t))$ is continuous in $t\in [0,1]$ (which then implies that it is globally constant). 
Fix $t_0 \in [0,1]$. Let $Y$ be an open set such that $K(f_{t_0})\subset Y\Subset U_{t_0}$. Then in a small neighborhood $T$ of $t_0$ in $[0,1]$, $f_t$ represents a GPL germ in $j_Y\mathcal{Q}_Y$ and $d(j_Y f_t)=d(j_Y f_{t_0})$. For $t\in T$, $t\mapsto f_t|_Y$ is continuous in $\mathcal{Q}_Y$. By Theorem~\ref{thm:continuity}, $t\mapsto \text{dim}(E^{vert}_*(f_t,U_t))$ is continuous at $t_0$.
\end{proof}
\subsection{A counterexample}\label{subsec:counterex}
We give an example of $f\in \mathcal{Q}_{W_0}$ with $D(f,W_0)=1$ and $E^{vert}_*(f,W_0)=0$. Let $W_0=\{z\in \C:|z|<4\}$ and $W=\{z\in \C: |z|<2\}$. Let $f_t(z)=z^2+t \sqrt{4-z}$, $z\in W_0$. 
\begin{proposition}\label{prop:counterex}
When $t\not=0$ and $|t|$ is small enough, $E^{vert}_*(f_t, W_0)=0$. 
\end{proposition}
\begin{proof}
Clearly, $f_0\in \mathcal{Q}_{W_0}$ and $K_*(f_0)=\mathbb{D}$. When $|t|$ is small, $f_t\in \mathcal{Q}_{W_0}$ and $K_*(f_t)\in W$. 
Let us assume for a contradiction that there is a sequence $t_n\to 0$ such that $E^{vert}_*(f_{t_n}, W_0)$ has a non-zero element $v_n$, $t_n\not=0$. Without loss of generality, assume that $\|v_n\|_{W}=1$. Let $\gamma_n:\C\setminus K_*(f_{t_n})\to \C$ be the holomorphic function with $\gamma_n(z)\to 0$ as $z\to \infty$ and such that 
\begin{equation}\label{eqn:gamman}
\gamma_n(f_{t_n}(z))-f_{t_n}'(z) \gamma_n(z)=v_n(z)
\end{equation}
for $z\in W_0\setminus K_*(f_{t_n})$. As in Lemma~\ref{lem:vertcomp}, after passing to a subsequence, we may assume that $v_n$ converges in $\mathcal{B}_{W}$ to an element $v\in E^{vert}_*(f_0, W)$ with $\|v\|_{W}=1$ and that $\gamma_n$ converges locally uniformly in $\C\setminus K_*(f_0)$ to a holomorphic function $\gamma:\C\setminus \overline{\mathbb{D}}\to \C$ with $\gamma(z)\to 0$ as $z\to \infty$ and such that 
\begin{equation}\label{eqn:vgamma}
v(z)=\gamma(f_0(z))-f_0'(z) \gamma(z)
\end{equation} for $z\in \C\setminus \overline{\mathbb{D}}$. 
Since $f_0(z)=z^2$ is holomorphic in $\C$, $v(z)$ extends to a holomorphic function in $\C$. Since $\gamma(z)=O(1/z)$ near infinity, $v(z)=O(1)$ near infinity. This means that $v=c$ is a constant. Hence $v\equiv 1$. Using (\ref{eqn:vgamma}), we obtain
$$\gamma(z)=-\sum_{n=1}^\infty ((f^n)'(z))^{-1}=-\sum_{n=1}^\infty (2^n z^{2^n-1})^{-1}.$$ 
In particular, $\gamma(4)\not=0$. Thus for $n$ large enough, $\gamma_n(4)\not=0$. 
As $\lim_{z\to 4}f_{t_n}'(z)=\infty$, by (\ref{eqn:gamman}), $\lim_{z\to 4} v_n(z)$ does not exist, a contradiction. 
\end{proof}

\section{The codimension of horizontal spaces}\label{sec:codim}
The goal of this section is to prove the following:
\begin{theorem}\label{thm:dimension}
For $f_0\in \mathcal{Q}_{W_0}$,  the horizontal space $E^h(f_0,W_0)$ has codimension $d(j_{W_0}f_0)$ in $\mathcal{B}_{W_0}$. 
\end{theorem}

The theorem was proved in \cite{L2} for a quadratic-like mapping.  For a GPL, the theorem was proved in~\cite{Smania} assuming that $f$ can be approximated by critically finite hyperbolic mappings. Whether a GPL can be approximated by hyperbolic mappings is a difficult problem, which is equivalent to the well-known Fatou conjecture (for polynomials).

Our proof of Theorem~\ref{thm:dimension} consists of three steps. First, we prove the theorem in \S~\ref{subsec:poly} in the case that $f_0$ is a restriction of a polynomial map with the same filled Julia set (and without critical relation in the basin of infinity), by applying a result in~\cite{Astorg}. By Corollary~\ref{cor:Dfgerm}, this implies the theorem in the case that $j_{W_0}f_0$ is qc conjugate to a restriction of polynomial. Unfortunately, this is not always the case. To deal with the difficulty, we introduce a notion {\em subordination} and show in \S~\ref{subsec:subord} that any GPL is subordinate to a {\em balanced } GPL so that their horizontal spaces have the same codimension. In \S~\ref{subsec:straighten}, we shall extend a straightening theorem in~\cite{Baranski} to show that a balanced GPL is hybrid equivalent to a polynomial and completes the proof of the theorem. 

\subsection{The polynomial case}\label{subsec:poly}

\begin{theorem}\label{thm:poly} Assume $f\in \mathcal{Q}_W$ extends to a polynomial $F:\mathbb{C}\to\mathbb{C}$ 
such that the following hold:
\begin{itemize}
\item $K_*(f)=K(F)$, where $K(F)$ is the filled Julia set of $F$, i.e. the set of $z\in \C$ with $\sup_{n=0}^\infty |F^n(z)|<\infty$;
\item $F$ has no critical relation in $\mathbb{C}\setminus K(F)$. That is, each critical point of $F$ in $\C\setminus K(F)$ is non-degenerate; and for critical points $c,c'$ of $F$ in $\C\setminus K(F)$ and non-negative integers $n, n'$, $F^n(c)=F^{n'}(c')$ implies that $n=n'$ and $c=c'$. 
\end{itemize}
Then 
$$\text{dim}(E^{vert}_*(f,W))=d(j_Wf),$$
where $d(j_W f)$ is defined in (\ref{eqn:criticality}). 
\end{theorem}


The proof of the theorem is broken into the following three lemmas. Let $\mathcal{P}_d$ denote the set of monic centered polynomials of degree $d$, i.e. polynomials of the form 
$z^d+a_{d-2}z^{d-2}+\dots+a_1z+a_0$, with $a_0,a_1,\dots, a_{d-2}\in \C$.  
First of all, without loss of generality, we may assume that $F\in \mathcal{P}_d$.
Write $d(f)=d(j_Wf)$.
\begin{lemma}\label{lem:Fvertup} Each $v\in E^{vert}_*(f,W)$ is a restriction of a polynomial of degree at most $d-2$.  
\end{lemma}
\begin{proof}
By definition, there exists a holomorphic function $\gamma:\mathbb{C}\setminus K(F)\to \mathbb{C}$ such that $\gamma(z)\to 0$ as $z\to \infty$ and 
\begin{equation}\label{eqn:vF}
\gamma(F(z))-F'(z)\gamma(z)=v(z)
\end{equation}
for all $z\in W\setminus K(F)$. Since the left hand side of (\ref{eqn:vF}) defines a holomorphic function in $\C\setminus K(F)$, $v$ extends to a holomorphic function in $\C$ such that (\ref{eqn:vF}) remains true for all $z\in  \C\setminus K(F)$. As $z\to\infty$, $\gamma(z)=O(1/z)$, $\gamma(F(z))=O(z^{-d})$, thus $v(z)=o(z^{d-1})$. Thus $v(z)$ is a polynomial of degree at most $d-2$. 
\end{proof}

\begin{lemma} $\text{dim}(E^{vert}_*(f,W))\ge d(f)$.
\end{lemma}
\begin{proof}
Let $v$ be a polynomial of degree at most $d-2$. Take $R>0$ large enough such that $|F(z)|>|z|$ and $|DF(z)|>2$ whenever $|z|>R$. Then 
$$\gamma(z)=-\sum_{n=1}^\infty \frac{v(F^{n-1}(z))}{(F^{n})'(z)}$$
is a holomorphic function in $\{|z|>R\}$, since the series on the right hand side converges uniformly. Obviously, 
\begin{equation}\label{eqn:vertF}
v(z)=\gamma(F(z))-F'(z)\gamma(z)
\end{equation}
whenever $|z|>R$. 

If $F$ has no critical point in $\mathbb{C}\setminus K(F)$, then using the functional equation $\gamma(z)=(\gamma(F(z))-v(z))/F'(z)$ repeatedly, we obtain analytic continuation of $\gamma$ to $\mathbb{C}\setminus K_*(f)$ such that (\ref{eqn:vertF})holds for all $z\in \C\setminus K(F)$. So in this case, $v\in E^{vert}_*(f)$.

Now consider the case that $F$ has critical points in $\mathbb{C}\setminus K(F)$. Using the same functional equation, we obtain analytic continuation of $\gamma$  to the complement of the backward orbits of critical points in $\mathbb{C}\setminus K(f)$. To extend $\gamma$ further to $\mathbb{C}\setminus K(F)$, it is sufficient to have 
$$\sum_{n=0}^\infty \frac{v(F^n(c))}{DF^n(F(c))}=0$$
for each critical point $c$ in $\mathbb{C}\setminus K(F)$, since $F$ has no critical relation in $\C\setminus K(F)$.
Since $F$ has $d-1-d(f)$ critical points in $\mathbb{C}\setminus K(F)$,
this shows that $E^{vert}_*(f)$ has dimension at least $d(f)$.
\end{proof}
\begin{lemma}
$\text{dim}(E^{vert}_*(f, W))\le d(f)$. 
\end{lemma}
\begin{proof} We shall use Astorg's solution~\cite{Astorg} to a problem in the dynmaical Teichm\"uller theory developed in \cite{McS}. Let 
$\mathcal{H}_0(F)$ denote the space of $G\in \mathcal{P}_d$ for which there exists a qc map $\varphi:\mathbb{C}\to\mathbb{C}$ such that $G\circ \varphi=\varphi\circ F$ on $\mathbb{C}$ and such that $\bar{\partial }\varphi=0$ a.e. on $K(F)$. Astorg's result implies that $\mathcal{H}_0(F)$ contains an embedded  submanifold $\mathcal{H}'$ of $\mathcal{P}_d$ of dimension $d-1-d(f)$ with $\mathcal{H}'\ni F$. Each tangent vector $v$ of $\mathcal{H}'$ at $F$ is contained in  $E^h(F, W)\cap \mathcal{P}_{d-2}=E^h(f,W)\cap \mathcal{P}_{d-2}$. By Theorem~\ref{thm:codimEh0},
$$\mathcal{P}_{d-2}= (E^h(f,W)\cap\mathcal{P}_{d-2})\oplus(E^{vert}_*(f,W)\cap\mathcal{P}_{d-2}).$$
By Lemma~\ref{lem:Fvertup}, it follows that 
 $$\dim(E^{vert}_*(f,W))=\text{dim}(E^{vert}_*(f,W)\cap\mathcal{P}_{d-2})\le d(f).$$  

Let us give more details on Astorg's result used above. Let $Bel(F)$ denote the space of $F$-invariant Beltrami differentials $\mu$ with $\|\mu\|_\infty<1$. Fix $Z=\{0,1,\infty\}$ and define a map 
$$\Psi^Z: Bel(F)\to Poly_d$$ as $\mu\mapsto \phi^Z_\mu\circ F\circ (\phi^Z_\mu)^{-1}$, where $\phi^Z_\mu$ is the unique qc map from the Riemann sphere $\widehat{\mathbb{C}}$ to itself fixing $0,1,\infty$ and $Poly_d$ is the space of polynomials of degree $d$. By the Measurable Riemann Mapping Theorem, $\phi^Z_\mu$ depends on $\mu$ holomorphically, and thus the map $\Psi^Z$ is holomorphic. Let $Bel_0(F)=\{\mu\in Bel(F):\mu=0 \mbox{ in } K(F)\}$. Then $\Psi^Z|_{Bel_0(F)}$ is also holomorphic. In \cite[Theorem 4.5]{Astorg}, the rank of $D\Psi^Z(0)$ is calculated. The proof shows that the rank of $D\Psi^Z|_{Bel_0(F)}(0)$ is equal to the number of grand orbits of critical points in $\mathbb{C}\setminus K(F)$, which is $d-1-d(f)$ in our case. By \cite[Proposition 7]{Astorg}, $D\Psi^Z|_{Bel_0(F)}(\mu)$ has the same rank for all $\mu\in Bel_0(F)$. For each $\mu\in Bel_0(F)$, $\Psi^Z(\mu)$ is affinely conjugate to a map in $\mathcal{H}_0(F)$. Since the coefficients of $\Psi^Z_\mu$ depend on $\mu$ holomorphically, in a neighborhood of $\mu=0$, we can reparametrize $\Psi^Z(\mu)$ to a polynomial $\Psi(\mu)\in \mathcal{P}_d$ which also depends holomorphically in $\mu$. The rank of $D\Psi|_{Bel_0}(\mu)$ keeps the same. By the Implicit Function Theorem, this implies the existence of $\mathcal{H}'$.       
\end{proof}

\begin{remark} The second condition is not necessary for Theorem~\ref{thm:poly} to hold. It was added for simplicity of proof. 
\end{remark}

\subsection{Subordination}\label{subsec:subord}
We say that a GPL mapping $f:U\to V$ is subordinate to another GPL mapping $\widetilde{f}:\widetilde{U}\to \widetilde{V}$ if the following hold:
\begin{itemize}
\item $\widetilde{V}=V$ and each component of $U$ is a component of $\widetilde{U}$; 
\item $\widetilde{f}|_U=f$;
\item for each component $Q$ of $\widetilde{U}\setminus U$, $Q$ is disjoint from the critical orbits of $f$ and $\widetilde{f}|_Q$ is conformal onto its image. 
\end{itemize}


\begin{lemma}\label{lem:subordJulia} If a GPL $f$ is subordinate to another GPL $\widetilde{f}$, then 
$$\text{area}\left(K(\widetilde{f})\setminus\bigcup_{n=0}^\infty \widetilde{f}^{-n}(K(f))\right)=0.$$
\end{lemma}
\begin{proof} Let $Y=\widetilde{U}\setminus U$ and consider the first return map $R:Y'\to Y$, where  
$$Y':=\{z\in Y: \exists n>0, \widetilde{f}^n(z)\in Y\},$$ and 
$R(z)=\widetilde{f}^{n(z)}(z)$, where $n(z)$ is the minimal positive integer such that $\widetilde{f}^{n(z)}(z)\in Y$. Let $$E=\{z\in Y': R^k(z)\in Y' \text{ for each } k=1,2,\dots\}.$$
Since each point in 
$K(\widetilde{f})\setminus\bigcup_{n=0}^\infty \widetilde{f}^{-n}(K(f))$ returns to $Y$ infinitely often, it suffices to show $\text{area}(E)=0$.

Since $Y$ is disjoint from the critical orbits of $\widetilde{f}$, $R$ maps each component of $Y'$ conformally onto a component of $Y$. Note also that $Y'\Subset Y$. So for each component $A$ of $Y$, $\text{area}(A\setminus Y')/\text{area}(A)>0$. For each $n=0,1,\ldots$, let $Y_n'=R^{-n}(Y)$. Then each component $Q$ of $Y_n'$ is mapped by $R^n$ onto a component $A$ of $Y$ and $R^{n+1}(Q)\subset Y'$. By the Koebe distortion theorem, $\text{area}(Q\setminus Y'_{n+1})/\text{area}(Q)$ is bounded from below by a positive constant $\varepsilon$. 
Thus $\text{area}(Y_{n+1}')/\text{area}(Y_n')\le 1-\varepsilon$. 
Since $E=\bigcap_{n=1}^\infty Y_n'$, this implies that $\text{area}(E)=0$. 
\end{proof}

A GPL $f:U\to V$ is non-escaping if $f$ has no critical point in $U\setminus K(f)$.
\begin{proposition}\label{prop:subord} If a non-escaping GPL $f:U\to V$ is subordinate to another GPL $\widetilde{f}:\widetilde{U}\to V$, then 
$\text{dim}(E^{vert}_*(f, U))=\text{dim}(E^{vert}_*(\widetilde{f},\widetilde{U})).$
\end{proposition}
\begin{proof} Let us consider the natural embedding $\iota: \mathcal{B}(\widetilde{U})\to \mathcal{B}(U)$. 

Since $K(f)\subset K(\widetilde{f})$, it follows from definition that $\iota(E^h(\widetilde{f}, \widetilde{U}))\subset E^h(f, U)$. 
Thus $\iota$ induces a linear map $$\bar{\iota}:\mathcal{B}(\widetilde{U})/E^h(\widetilde{f}, \widetilde{U})\to \mathcal{B}(U)/E^h(f, U).$$
By Lemma~\ref{lem:Lyudecomp}, it suffices to show that $\bar{\iota}$ is a linear automorphism. Since $\iota$ is surjective, so is $\bar{\iota}$. It remains to show that $\bar{\iota}$ is injective. This amounts to prove the following.

{\bf Claim.} Let $v\in \mathcal{B}(\widetilde{U})$. Assume that there is a qc vector field $\alpha(z)/dz$ in $\mathbb{C}$ such that 
$$v(z)=\alpha(f(z))-f'(z) \alpha(z)$$
holds in a neighborhood $P$ of $K(f)$ in $U$ and $\bar{\partial} \alpha=0$ holds a.e. on $K(f)$. Then $v\in E^h(\widetilde{f}, \widetilde{U})$. 

To prove this, we shall apply the infinitesimal pullback argument by Avila-de Melo-Lyubich~\cite[Section 6.2]{ALM}. Let $\{P_n\}_{n=0}^\infty$ be a sequence of neighborhoods of $K(f)$ such that $P_0=P\subset U$ and $\widetilde{f}(P_{n})\Subset P_{n-1}\cap \widetilde{U}$ for each $n\ge 1$.
We shall construct inductively a sequence $\beta_n(z)/dz$ qc vector fields on $\mathbb{C}$, $n=0,1,\ldots$, with the following properties:
\begin{enumerate}
\item $\beta_n(z)=0$ outside $V$;
\item $\beta_n(\widetilde{f}(z))-\widetilde{f}'(z)\beta_n(z)=v(z)$ for $z\in \partial \widetilde{U}$;
\item $\beta_n(z)=\alpha(z)$ for $z\in P_n$; 
\item for $n\ge 1$,  $v(z)=\beta_{n-1}(\widetilde{f}(z))-\widetilde{f}'(z)\beta_n(z)$,  
if $z\in \widetilde{U}\setminus P_{n}$.
\end{enumerate}

For $n=0$, take an arbitrary $\beta_0$ such that (1)-(3) holds. Assume now that $n\ge 1$ and $\beta_{n-1}(z)/dz$ has been constructed such that (1)-(3) holds with $n$ replaced by $n-1$. 
Define
$$\beta_{n}(z)=\left\{
\begin{array}{ll}
\beta_{n-1}(z) & \mbox{ if } z\not\in \widetilde{U};\\
\frac{\beta_{n-1}(\widetilde{f}(z))-v(z)}{\widetilde{f}'(z)},  
&\mbox{ if }z\in \widetilde{U}\setminus P_{n}.\\
\alpha(z) & \mbox{ if } z\in P_{n}.
\end{array}
\right.
$$
Since $f$ is non-escaping, $\text{Crit}(\widetilde{f})=\text{Crit}(f) \subset K(f)\subset P_{n+1}$. Thus $\beta_{n}$ is well-defined and satisfies (1)-(3). Moreover, for $z\in P_{n}\setminus K(f)$, we have 
$$\beta_{n}(z)=\alpha(z)=\frac{\alpha(\widetilde{f}(z))-v(z)}{\widetilde{f}'(z)}
=\frac{\beta_{n-1}(\widetilde{f}(z))-v(z)}{\widetilde{f}'(z)},$$
which verifies (4). The vector field $\beta_n(z)/dz$ is clearly qc in $\C\setminus \partial P_n$. For $z$ in a small neighborhood $Q$ of $\partial P_n$, $\widetilde{f}(z)\in P_{n-1}\setminus K(\widetilde{f})$, so $\beta_n(z)=\alpha(z)$. This shows that $\beta_n(z)/dz$ is qc in $Q$, hence it is qc in $\C$.
This completes the construction of $\beta_n(z)$. 

Note that for each $n\ge 1$, $\beta_{n}(z)/dz$ is a qc vector field and $$\|\bar{\partial} \beta_{n}\|_\infty= \max(\|\bar{\partial} \beta_{n-1}\|_\infty, \|\bar{\partial} \alpha\|_\infty)=\max(\|\bar{\partial} \beta_0\|_\infty, \|\bar{\partial} \alpha\|_\infty).$$
Moreover, $\beta_n=\beta_0$ outside $\widetilde{U}$ and $\beta_n=\alpha$ on $K(f)$,  $\sum_{k=0}^{n-1}\beta_k(z)/n$ has a subsequence which converges uniformly in $\mathbb{C}$ to a qc vector field $\gamma$ with $\gamma=\alpha$ on $K(f)$. Then 
$$v(z)=\gamma(\widetilde{f}(z))-\widetilde{f}'(z) \gamma(z)$$
for all $z\in \widetilde{U}$. 

It remains to show that $\bar{\partial} \gamma(z)=0$ for a.e. $z\in K(\widetilde{f})$. Indeed, since $\gamma=\alpha$ on $K(f)$, we have $\bar{\partial} \gamma=\bar{\partial}\alpha=0$ holds a.e. on $K(f)$. Since $\bar{\partial}\gamma |_{K(f)}$ is $f$-invariant, we have 
$\bar{\partial}\gamma =0$ a.e. on $\bigcup_{n=0}^\infty \widetilde{f}^{-n}(K(f))$. 
By Lemma~\ref{lem:subordJulia}, $K(\widetilde{f})\setminus \bigcup_{n=0}^\infty \widetilde{f}^{-n}(K(f))$ has zero area. Thus  $\bar{\partial} \gamma(z)=0$ for a.e. $z\in K(\widetilde{f})$.
\end{proof}

\subsection{Straightening}\label{subsec:straighten}
We say that a GPL $f:U\to V$ is {\em balanced}, if the following holds for each component $Y$ of $V$,
\begin{itemize}
\item there exists a positive integer $d_{Y}$ such that for each $w\in V$, $f(z)-w$ has exactly $d_Y$ roots in $U\cap Y$, counting multiplicity; 
\item letting $m_Y$ denote the number of components of $Y\cap U$, then 
\begin{equation}\label{eqn:RH}
r_Y=m_Y-1- d_Y(n-1)
\ge 0,
\end{equation}
where $n$ is the number of components of $V$.
\end{itemize}

\begin{lemma}\label{lem:subbal}
Each GPL $f:U\to V$ is subordinate to a balanced GPL $\widetilde{f}: \widetilde{U}\to \widetilde{V}$ with $\widetilde{V}=V$, and $\text{Crit}(f)=\text{Crit}(\widetilde{f})$. 
\end{lemma}
\begin{proof} We simply add finitely many conformal branches to the domain of $f$. Note that the critical orbits intersect $V\setminus \overline{U}$ only at finitely many points. 
\end{proof}

\begin{theorem}\label{thm:straighten}
Let $f:U\to V$ be a non-escaping balanced GPL. Then $f$ is hybrid conjugate to a polynomial $P$ with a conjugacy sending $K(f)$ to $K(P)$. Moreover, we can choose $P$ such that it has no critical relations in $\mathbb{C}\setminus K(P)$.
\end{theorem}
\begin{proof} Let $d=\sum_{Y} d_Y$, where $Y$ runs over all components of $V$. 
Let $R>1$ be a large number such that $V'':=\mathbb{D}_{R^d}\Supset V':=\mathbb{D}_{R}$ contains $\overline{V}$. 
We extend $f$ to a smooth branched covering from $\mathbb{C}$ to $\mathbb{C}$ as follows. 

{\bf Step 1.} Define $f(z)=z^d$ for all $z\not\in V'$. 

{\bf Step 2.} Extend the definition of $f$ in $V'\setminus V$ so that $f: V'\setminus V\to V''\setminus V'$ is a branched covering of degree $d$, continuously extendable to boundary, such that for each component $Y$ of $V$, $f|_{\partial Y}$ is a $d_Y$-to-1 unbranched smooth covering onto $\partial V'$. For this step, by the Riemann-Hurwitz formula, we need to add $n-1$ critical points (counting multiplicity) in $V'\setminus \overline{V}$, where $n$ is the number of components of $V$. We choose the newly added critical points to be simple and have disjoint forward orbits under $f$ (already defined on $\C\setminus V$). 

{\bf Step 3.} We extend the definition of $f$ to $Y\setminus U$ for each component $Y$ of $V$, so that $f:Y\setminus U\to V'\setminus V$ is a branched covering of degree $d_Y$. The condition (\ref{eqn:RH}) guarantees that such an extension exists by adding $r_Y$- simple critical points with disjoint orbits. We can choose the location of these newly added critical points such that their orbits are disjoint from the orbit of critical points added in Step 2. 

The map $f$ is quasiregular, and it is holomorphic outside $\overline{V'\setminus U}$. Any orbit of $f$ can intersect $\overline{V'\setminus U}$ at most twice. Thus $f$ is hybrid conjugate to a polynomial $P$.

\end{proof}
\begin{remark} In~\cite[Corollary 9]{Baranski}, the theorem is proved under an additional technical assumption $r_Y=0$ for each component $Y$ of $V$. If such an assumption holds, then in Step 3 above, no critical point is needed to add.
\end{remark}

\begin{proof}[Proof of Theorem~\ref{thm:dimension}] 
Let $f_0:U\to V$ be a GPL representative of $j_{W_0}f_0$ such that $f_0$ has no critical point in $\overline{U}\setminus K_*(f_0)$. By Theorem~\ref{thm:codimEh0}, it suffices to show that $\text{dim}(E^{vert}_*(f_0,U))=d(j_{W_0}f_0)$. 
By Lemma~\ref{lem:subbal} and Proposition~\ref{prop:subord}, we may assume that $f_0:U\to V$ is balanced. 

By Theorem~\ref{thm:straighten}, $f_0$ is hybrid equivalent to a polynomial $P:U'\to V'$ such that $P$ has no critical relation in $\mathbb{C}\setminus K(P)$ via a conjugacy sending $K_*(f_0)$ to $K(P)$. In particular, 
$$d(j_{U'}P)=d(j_U f_0)=d(j_{W_0}f_0).$$
By Corollary~\ref{cor:Dfgerm}, $\text{dim}(E^{vert}_*(f_0,U))=\text{dim}(E^{vert}_*(P, U'))$. 
By Theorem~\ref{thm:poly}, $$\text{dim}(E^{vert}_*(P,U'))=d(j_{U'}P)=d(j_{W_0}f_0).$$ 
\end{proof}
\section{Hybrid classes in Banach neighborhoods}\label{sec:hybridclass}
In this section, we shall prove Theorem~\ref{thm:hybridclass}. The main idea is to view the connected component of a local hybrid class as the set of fixed points of the `mating' map $I^f(g)$. Complication arises as the domain and range of $I^f$ lie in different Banach spaces.  

We start with a variation of the Implicit Function Theorem which will be used in our proof. 

For a subset $E$ of a complex Banach space $\mathcal{B}$ and $f_0\in E$, denote by $\mathbb{T}(E, f_0)$ the set of all $v\in \mathcal{B}$ for which there exists a sequence $f_n\in E\setminus \{f_0\}$ and $t_n\in \C$ such that $f_n\to f_0$ and 
$$t_n(f_n-f_0)\to v.$$
Note that when $E$ is an embedded submanifold of $\mathcal{B}$, $\mathbb{T}(E,f_0)$ is the tangent space of $E$ at $f_0$.

\begin{lemma}\label{lem:holomorphiczero} Let $\tau:\mathcal{W}\to \mathcal{B}_2$ be a complex analytic map with $\tau(0)=0$, where $\mathcal{W}$ is a neighborhood of the origin in a complex Banach space $\mathcal{B}_1$, and $\mathcal{B}_2$ is another complex Banach space. Assume that the following hold:
\begin{enumerate}
\item $\mathcal{B}_j$ can be written as the direct sum of two closed subspaces $\mathcal{B}_j=E_j\oplus F_j$, $j=1,2$, with $E_1=Ker(d\tau_0)$ and $F_2= Im(d\tau_0)$, and $\text{dim}(F_1)<\infty$;
\item $d(q\circ\tau)_0: F_1\to F_2$ is invertible, where $q: \mathcal{B}_2\to F_2$ is the natural projection;
\item $p(\mathbb{T}(\tau^{-1}(0),0))$ is dense in $E_1$, where $p:\mathcal{B}_1\to E_1$ is the natural projection.
\end{enumerate}
Then $\tau^{-1}(0)$ is an embedded complex analytic submanifold of $\mathcal{B}_1$ of codimension $\text{dim}(F_1)$ near $0$. 
\end{lemma}
\begin{proof} Identify $\mathcal{B}_j$ with the product $E_j\times F_j$, $j=1,2$. 
Applying the Implicit Function Theorem to the map $$q\circ \tau:E_1\times F_1\to F_2,$$ 
we obtain a holomorphic map $\sigma: U\to V$, where $U$ and $V$ are small balls centered at the origin in $E_1$ and $F_1$ respectively, such that for each $u\in U$, $\sigma(u)$ is the only point in $V$ satisfying $q\circ \tau (u,\sigma(u))=0$. 

Put $$\widetilde{\tau}(u):=\tau(u,\sigma(u)).$$ Let us show
\begin{equation}\label{eqn:tildetau}
\mathbb{T}(\widetilde{\tau}^{-1}(0),0)\supset p(\mathbb{T}(\tau^{-1}(0),0)).
\end{equation}

For $w\in \mathbb{T}(\tau^{-1}(0),0)$,
let $w_n\in \tau^{-1}(0)\setminus \{0\}$, $t_n\in \C$ be such that 
$w_n\to 0$ and $t_n w_n\to w$. Since $\tau^{-1}(0)\subset (q\circ \tau)^{-1}(0)$, $w_n\in (q\circ \tau)^{-1}(0)$. Thus $w_n$ can be written in the form $(u_n, \sigma(u_n))$ for some $u_n\in E_1$, provided that $n$ is large enough. 
Thus $u_n\in \widetilde{\tau}^{-1}(0)$ and $p(w)=\lim_{n\to\infty} t_nu_n\in E_1$. This shows that $p(w)\in \mathbb{T}(\widetilde{\tau}^{-1}(0),0)$.
%

We shall next show that 
$\widetilde{\tau}(u)=0$ holds for all $u\in U$. 
This will conclude the proof of the lemma since it implies 
$$\tau^{-1}(0)\cap (U\times F_1)=\{(u,\sigma(u)): u\in U\}.$$ 

Assume for a contradiction that $\widetilde{\tau}$ is not identically zero in $U$. Let $B(r)$ denote the ball in $E_1$ centered at $0$ of radius $r$. Let $r_0>0$ be a small number such that $U\supset B(3 r_0).$ Then for each $u\in B(r_0)$, by Cauchy's integral formula, we have 
$$\widetilde{\tau}(u)=\frac{1}{2\pi i} \int_{|\lambda|=2} \frac{\widetilde{\tau}(\lambda u)}{\lambda-1} d\lambda=\sum_{k=0}^\infty P_k(u),$$  
where 
$$P_k(u)=\frac{1}{2\pi i} \int_{|\lambda|=2} \lambda^{-k-1} \widetilde{\tau}(\lambda u) d\lambda.$$ 
Note that 
\begin{equation}\label{eqn:homoge}
P_k(tu)=t^k P_k(u)
\end{equation}
holds for all $u\in B(r_0)$ and $t\in \mathbb{D}$. It follows that $P_k$ extends to a holomorphic function in $E_1$ such that (\ref{eqn:homoge}) holds for all $u\in E_1$ and each $t\in \C$. In particular, $P_0(u)=P_0(0)=\widetilde{\tau}(0)=0$. Let $C$ be a positive constant such that $\|\widetilde{\tau}(u)\|\le C$ for all $u\in B(2r_0)$. Then $\|P_k(u)\|\le C$ for each $k=0,1,\ldots$ and $u\in B(r_0)$. By (\ref{eqn:homoge}), it follows that 
$$\|P_k(u)\|\le \frac{C}{r_0}\|u\|^k.$$
Since $\widetilde{\tau}$ is not identically zero, there is a minimal positive integer $k_0$ such that $P_{k_0}(u)$ is not identically zero, and $$\widetilde{\tau}(u)=P_{k_0}(u)+R(u),$$
where $$R(u)=\sum_{k=k_0+1}^\infty P_k(u)=O(\|u\|^{k_0+1}), \mbox{ as } u\to 0.$$

{\bf Claim.} For any $u\in \mathbb{T}(\widetilde{\tau}^{-1}(0),0)$, $P_{k_0}(u)=0$.

By definition, there exists $(u_n)_{n=1}^\infty \subset E_1\setminus \{0\}$ and $t_n\in \C$ such that $u_n\to 0$, $\widetilde{\tau}(u_n)=0$ and $t_n u_n\to u$. 
Then for all $n$ large, $$P_{k_0}(u_n)=\widetilde{\tau}(u_n)-R(u_n)=-R(u_n)=O(\|u_n\|^{k_0+1}).$$
By(\ref{eqn:homoge}), this implies that  
$$P_{k_0}\left(t_nu_n\right)=O(|t_n|^{k_0}\|u_n\|^{k_0+1})=O(\|t_nu_n\|^{k_0} \|u_n\|)=O(\|u_n\|).$$
Since $u_n\to 0$, by continuity,
$$P_{k_0}\left(u\right)=0.$$

Since $p(\mathbb{T}(\tau^{-1}(0),0)$ is dense in $E_1$, by (\ref{eqn:tildetau}), $\mathbb{T}(\widetilde{\tau}^{-1}(0),0)$ is dense in $E_1$. It follows from the claim that $P_{k_0}(u)=0$ for all $u\in E_1$, contradicting the fact that $P_{k_0}$ is not identically zero. 
\end{proof}

The `mating' map is used in the proof of the following lemma. 
\begin{lemma}\label{lem:mating0}
Let $f_0\in\mathcal{Q}_{W_0}$ for some open set $W_0\subset \C$. For any open subset $W$ of $\C$ with $K_*(f_0)\subset W\Subset W_0$, there is an open subset $\widetilde{W}$ of $\C$
with $K_*(f_0)\subset \widetilde{W}\Subset W$ and a neighborhood $\mathcal{V}$ of $f_0|_W$ in $\mathcal{B}_W$ such that 
the following hold:
\begin{itemize}
\item $\widetilde{W}$ is a finite union of Jordan disks with pairwise disjoint closures in $\C$;
\item for any neighborhood $\widetilde{\mathcal{U}}$ of $f_0|_{\widetilde{W}}$ in $\mathcal{Q}_{\widetilde{W}}$, if $f\in \mathcal{Q}_{W_0}$, $f|_W\in \mathcal{V}$, $f|_{\widetilde{W}}\in \widetilde{\mathcal{U}}$ and $d(j_{W}f)=d(j_{W_0}f_0)$, then 
$\mathbb{T}(\mathcal{H}_*(f,\widetilde{\mathcal{U}}), f|_{\widetilde{W}})$ is dense in $E^h(f, \widetilde{W})$;
\item  
    if $f\in \mathcal{Q}_{W_0}$, $f|_W\in \mathcal{V}$, $d(j_{W_0}f)=d(j_{W_0}f_0)$ and $\mathbb{T}(\mathcal{H}_*(f,\mathcal{V}), f|_W)$ is dense in $E^h(f,W)$, then $\mathcal{H}_*(f|_W,\mathcal{V})$ is an embedded complex analytic submanifold of $\mathcal{B}_W$ with codimension $d(j_{W_0}f)$ near $f|_W$.
\end{itemize}
\end{lemma} 

\begin{proof} 
Let $f_0:U_{f_0}\to V_{f_0}$ be a GPL representative of $j_{W_0}f_0$ with $U_{f_0}\Subset W$ and let $\widetilde{W}=(f_0|_{U_{f_0}})^{-1} (U_{f_0})$. Then $\widetilde{W}$ is an open subset of $\C$ with $K_*(f_0)\subset \widetilde{W}\Subset U_{f_0}$ and the complement of the closure of $\widetilde{W}$ is connected in $\C$. Let $\mathcal{U}$ be a small neighborhood of $f_0$ in $\mathcal{Q}_W$ that can be equipped as in Definition~\ref{dfn:equipped}. So there is a holomorphic motion $h_g$ of $\C$ over $(\mathcal{U}, f_0|_W)$ such that for each $g\in \mathcal{U}$, $U_g:=h_g(U_{f_0}),$ $V_g:=h_g(V_{g_0})$, $g:U_g\to V_g$ is a representative of $j_{W}g$ and $h_g\circ f_0=g\circ h_g$ on $\partial U_{f_0}$.

We shall use the `mating' procedure in~\cite{DH2} to construct a GPL $$i^f(g): U'_g\to V'_g$$ for each $f,g\in \mathcal{U}$. The procedure goes as follows.
Let $\nu_{f,g}$ denote the complex dilatation of the qc map $h_g\circ h_f^{-1}|_{\C\setminus U_f}$ and define a Beltrami differential $\mu_g$ in $\mathbb{C}$ by taking
\begin{itemize}
\item $\mu_{f,g}=\nu_{f,g}$ outside $U_f$;
\item $\mu_{f,g}=(f^n)^*(\nu_{f,g})$ on $(f|_{U_f})^{-n}(V_{f}\setminus U_f)$; 
\item $\mu_{f,g}=0$ on $K_*(f)$.
\end{itemize}
Then $f^*(\mu_{f,g}|_{V_{f}})=\mu_{f,g}|_{U_f}$.
Let $\varphi_{f,g}:\mathbb{C}\to \mathbb{C}$ be the qc map with complex dilatation $\mu_{f,g}$, normalized so that $\varphi_{f,g}(z)=z+o(1)$ near infinity, and let  
$$i^f(g)(z)=\varphi_{f,g}\circ f\circ (\varphi_{f,g})^{-1}, \text{ with }U'_{g}=\varphi_{f,g}(h_f(U_{f_0})), V'_{f,g}=\varphi_{f,g}(h_f(V_{f_0})).$$

In general, $U'_g\not\supset W$, so the map $i^f(g)$ is not well-defined on $W$. But by shrinking $\mathcal{U}$ to a smaller neighborhood $\mathcal{V}$, it is  well-defined on $\widetilde{W}$.
To see this, let $C_0>1$ be a constant such that for any $C_0$-qc map $\varphi:\C\to \C$ which is holomorphic in a neighborhood of $\infty$ and satisfies $\varphi(z)=z+o(1)$ near infinity, we have $\varphi(U_0)\Supset\widetilde{W}$. 
Let $\mathcal{V}\subset \mathcal{U}$ be a small ball in $\mathcal{B}_W$ centered at $f_0|_W$ such that the maximal dilatation of $h_g$ is less than $C_0^{1/3}$ for each $g\in \mathcal{U}$. Then for each $f,g\in \mathcal{U}$, $\varphi_{f,g}\circ h_f$ has maximal dilatation less than $C_0$, so $U'_{g}\Supset \widetilde{W}$. Therefore, 
$$I^f(g)=i(g)|_{\widetilde{W}}$$
defines a map from $\mathcal{V}$ into $\mathcal{B}_{\widetilde{W}}$ for each $f\in \mathcal{V}$. 

By the Measurable Riemann Mapping Theorem, $g\mapsto \varphi_{f,g}(z)$ is holomorphic for each $z\in \C$, so $I^f$ is holomorphic. 



Let us consider the set 
$$\mathcal{X}_f=\{g\in \mathcal{V}: I^f(g)=g|_{\widetilde{W}}\}.$$

{\bf Claim 1.} $\mathcal{H}_*(f, \mathcal{V})\subset \mathcal{X}_f$. 

Indeed, by Proposition~\ref{prop:connected set}, $\mathcal{H}_*(f,\mathcal{V})\subset \mathcal{H}_*(f,\mathcal{U})$ is contained in $\mathcal{H}_\#(f,\mathcal{U})$. 
So for each $g\in \mathcal{H}_*(f,\mathcal{V})$, $h_g\circ h_f^{-1}$ extends to a qc map $\psi$ which is a hybrid conjugacy between $f: U_f\to V_f$ and $g: U_g\to V_g$. Since the qc maps $\psi$ and $\varphi_{f,g}$ have the same complex dilatation 
and they are both normalized at infinity, we have $\psi\equiv \varphi_{f,g}$, which implies that $I^f(g)=g|_{\widetilde{W}}$. This proves the claim.

In particular, we have 
\begin{equation}\label{eqn:If0}
I^f(f)=f|_{\widetilde{W}}.
\end{equation}

Let us calculate the differential $D I_f$ at the point $g=f$. Assume $f\in \mathcal{Q}_{W_0}$ and $f|_W\in \mathcal{V}$. Let $P^h_f:\mathcal{B}_{W}\to E^h(f,W)$ denote the projection associate to the direct sum decomposition $\mathcal{B}_{W}=E^h(f,W)\oplus E^{vert}_*(f,W)$. Define $\widetilde{P}^h_f$ similarly with $W$ replaced by $\widetilde{W}$.

{\bf Claim 2.} Assume that $f\in\mathcal{Q}_{W_0}$, $f|_W\in \mathcal{V}$ and $d(j_Wf)=d(j_{W_0}f_0)$. Then for each $v\in \mathcal{B}_W$, 
\begin{equation}\label{eqn:derivative}
DI^f(f,v)=P^h_f(v)|_{\widetilde{W}}.
\end{equation}

To see this, let $$\alpha(z)=\left.\frac{d}{dt} (h_{f+tv}\circ h_f^{-1}(z))\right|_{t=0}$$
for $z\in \mathbb{C}\setminus U_f$, and let 
$$\beta(z)=\left.\frac{d}{dt}\varphi_{f,f+tv}(z)\right|_{t=0}$$
for each $z\in\C$. 
Then $\alpha(z)/dz$ is a qc vector field on $\mathbb{C}\setminus K_*(f)$ with $\alpha(z)=0$ for all $z\not\in V_f$ and $\alpha(f(z))-f'(z)\alpha(z)=v(z)$ for all $z\in \partial U_f$. 
Since $h_{f+tv}\circ h_f^{-1}$ and $\varphi_{f, f+tv}$ has the same complex dilatation on $\mathbb{C}\setminus U_f$, we have $\bar{\partial} \beta(z)=\bar{\partial} \alpha(z)$ for $z\in \mathbb{C}\setminus U_f$. Since $\varphi_{f, f+tv}$ is a hybrid conjugacy between $f:U_f\to V_f$ and $i^f(F_t)$, where $F_t=f+tv$, we have that $\bar{\partial }\beta(z)$ is $f$-invariant, 
and $\bar{\partial} \beta=0$ a.e. on $K_*(f)$. Since $d(j_{W}f)=d(j_{W_0}f_0)$, $f:U_f\to V_f$ is a GPL representative of $j_{W}f$ with no critical point in $\overline{U_f}\setminus K_*(f)$. By Lemma~\ref{lem:Lyudecomp},
we have $$P^h_f(v)(z)=\beta(f(z))-f'(z)\beta(z), \,\, z\in U_f.$$
Since $$DI^f_{f}(v)(z)=\left.\frac{d}{dt}\right|_{t=0} I^f(f+tv)(z)=\beta(f(z))-f'(z) \beta(z),$$
(\ref{eqn:derivative}) holds. 

(i) Fix a neighborhood $\widetilde{\mathcal{U}}$ of $f_0|_{\widetilde{W}}$ in $\mathcal{B}_{\widetilde{W}}$. Assume that $f\in \mathcal{Q}_{W_0}$, $f|_W\in \mathcal{V}$, $f|_{\widetilde{W}}\in \widetilde{\mathcal{U}}$ and $d(j_{W}f)=d(j_{W_0}f_0)$. For each $v\in E^h(f, W)$, there exists $r>0$ such that $g_t :=I^{f|_W}(f|_W+tv)$, $|t|<r$, is a holomorphic family in $\widetilde{\mathcal{U}}$.
So $g_t\in \mathcal{H}_*(f|_{\widetilde{W}}, \widetilde{\mathcal{U}})$. By Claim 2, 
$$\frac{g_t-g_0}{t}\to DI^{f_W}_{f|_W}(v)= v|_{\widetilde{W}}.$$ 
Thus $$\mathbb{T}(\mathcal{H}_*(f|_{\widetilde{W}},\widetilde{\mathcal{U}}), f|_{\widetilde{W}})\supset \{v|_{\widetilde{W}}: v\in E^h(f,W)\}.$$
By Lemma~\ref{lem:dense}, this set is dense in $E^h(f,\widetilde{W})$.

(ii) Now assume $f\in \mathcal{Q}_{W_0}$, $f|_W\in \mathcal{V}$ and $d(j_{W}f)=d(j_{W_0}f_0)$. Let us apply Lemma~\ref{lem:holomorphiczero} to the map 
$$\tau(v)=I^{f|_W}(f|_W+v)-(f+v)|_{\widetilde{W}}.$$
Put $\mathcal{B}_1=\mathcal{B}_{W}$, $\mathcal{B}_2=\mathcal{B}_{\widetilde{W}}$, $\mathcal{W}=\{g-f|_W:g\in \mathcal{V}\}$, $E_1=E^h(f,W), F_1=E^{vert}_*(f,W), E_2=E^h(f,\widetilde{W}), F_2=E^{vert}_*(f,\widetilde{W}).$
Then 
\begin{itemize}
\item The condition (1) follows from Theorem~\ref{thm:codimEh0}.
\item The condition (2) follows from Claim 2: By Claim 2,  $d(q\circ \tau) (v)=-v|_{\widetilde{W}}$ for each $v\in F_1$.
\item The condition (3) holds. By Claim 1, $\mathbb{T}(\tau^{-1}(0),0)$ contains $\mathcal{H}_*(f|_W, \mathcal{V})$, which is dense in $E_1$ by the assumption. 
\end{itemize}
So all the conditions in Lemma~\ref{lem:holomorphiczero} are satisfied,  which implies that $\tau^{-1}(0)$ is an embedded complex analytic submanifold of $\mathcal{B}_W$ of codimension $d(j_Wf)$ near $f|_W$. Locally $\tau^{-1}(0)$ coincide with $\mathcal{H}_*(f|_W, \mathcal{V})$ by definition of $\tau$ and Claim 1. 
\end{proof}

\begin{corollary}\label{cor:mate}
Let $f_0\in \mathcal{Q}_{W_0}$ for some open subset $W_0$ of 
$\C$. Then there is a neighborhood $\mathcal{V}_0$ of $f_0$ in $\mathcal{Q}_{W_0}$, an open subset $\widetilde{W}$ of $\C$ with $K_*(f_0)\subset \widetilde{W}\Subset W_0$ and a neighborhood $\widetilde{\mathcal{V}}$ of of $f_0|_{\widetilde{W}}$ in $\mathcal{Q}_{\widetilde{W}}$ with the following property: 
\begin{enumerate}
\item [(i)] $\widetilde{W}$ is a finite union of Jordan disks with pairwise disjoint closures in $\C$; 
\item [(ii)] if $f\in \mathcal{V}_{W_0}$ and $d(j_{W_0}f)=d(j_{W_0}f_0)$, then $f|_{\widetilde{W}}\in \widetilde{\mathcal{V}}$ and $\mathcal{H}_*(f|_{\widetilde{W}},\widetilde{\mathcal{V}})$ is a complex analytic submanifold of $\mathcal{B}_{\widetilde{W}}$ of codimension $d(j_{W_0}f_0)$ near $f_{\widetilde{W}}$.
\end{enumerate}
\end{corollary}
\begin{proof}  Take an open subset $W$ of $\C$ with $K_*(f_0)\subset W\Subset W_0$ arbitrarily.   Applying Lemma~\ref{lem:mating0}, we obtain a neighborhood $\mathcal{V}$ of $f_0|_W$ in $\mathcal{B}_W$ and an open subset of $\C$ with $K_*(f_0)\subset \widetilde{W}\Subset W$ such that (i) and the following (iii) hold: 
\begin{enumerate}
\item [(iii)] for any neighborhood $\mathcal{\widetilde{U}}$ of $f_0|_{\widetilde{W}}$, if $f\in \mathcal{Q}_{W_0}$, $f|_W\in \mathcal{V}$, $f|_{\widetilde{W}}\in \mathcal{\widetilde{U}}$ and $d(j_{W}f)=d(j_{W_0}f_0)$, then 
$\mathbb{T}(\mathcal{H}_*(f|_{\widetilde{W}},\widetilde{\mathcal{U}}), f|_{\widetilde{W}}$ is dense in $E^h(f, \widetilde{W})$. 
\end{enumerate}

Since $K_*(f_0)\subset \widetilde{W}\Subset W_0$, we may apply the lemma with $W$ replaced by $\widetilde{W}$ to obtain a neighborhood $\widetilde{\mathcal{V}}$ of $f_0|_{\widetilde{W}}$ in $\mathcal{B}_{\widetilde{W}}$ with the following property: 
\begin{enumerate}
\item[(iv)] If $f\in \mathcal{Q}_{W_0}$, $f|_{\widetilde{W}}\in \widetilde{\mathcal{V}}$ and $d(j_{\widetilde{W}}f)=d(j_{W_0}f_0)$ and $\mathbb{T}(\mathcal{H}_*(f|_{\widetilde{W}},\widetilde{\mathcal{V}}), f|_{\widetilde{W}})$ is dense in $E^h(f,\widetilde{W})$, then $\mathcal{H}_*(f|_{\widetilde{W}},\widetilde{\mathcal{V}})$ is a complex analytic submanifold of $\mathcal{B}_{\widetilde{W}}$ of codimension $d(j_{W_0}f)$ near $f$.  
\end{enumerate}
By Lemma~\ref{lem:K*cont}, there is a neighborhood $\mathcal{V}_0$ of $f_0$ in $\mathcal{Q}_{W_0}$ such that $f\in \mathcal{V}_0$ implies $K_*(f)\subset \widetilde{W}$. Shrinking $\mathcal{V}_0$ if necessary, we may assume $f\in \mathcal{V}_0$ implies that $f|_W\in \mathcal{V}$ and $f|_{\widetilde{W}}\in \widetilde{\mathcal{V}}$. 

Now assume $f\in \mathcal{V}_0$ and $d(j_{W_0}f)=d(j_{W_0}f_0)$. Since $K_*(f,W_0)\subset \widetilde{W}\subset W$, $d(j_{\widetilde{W}}f)=d(j_Wf)=d(j_{W_0}f)$.  
Applying (iii) to $\widetilde{\mathcal{U}}=\widetilde{\mathcal{V}}$, we obtain that $\mathbb{T}(\mathcal{H}_*(f|_{\widetilde{W}},\widetilde{\mathcal{V}}), f|_{\widetilde{W}})$ is dense in $E^h(f,\widetilde{W})$. Thus by (iv), $\mathcal{H}_*(f|_{\widetilde{W}},\widetilde{\mathcal{V}})$ is a complex analytic submanifold of $\mathcal{B}_{\widetilde{W}}$ of codimension $d(j_{W_0}f)$ near $f$.  
\end{proof}

\begin{proof} [Proof of Theorem~\ref{thm:hybridclass}]
Let $f_0\in \mathcal{Q}_{W_0}$ and let $d=d(j_{W_0}f_0)$. By Corollary~\ref{cor:mate}, there is a neighborhood $\mathcal{V}_0$ of $f_0$ in $\mathcal{Q}_{W_0}$, a neighborhood $\widetilde{W}$ of $K_*(f_0)$ in $\C$ with $\widetilde{W}\Subset W_0$ and a neighborhood $\widetilde{\mathcal{V}}$ of $f_0|_{\widetilde{W}}$ in $\mathcal{B}_{\widetilde{W}}$ such that if $f\in \mathcal{V}_{0}$ and $d(j_{W_0}f)=d(j_{W_0}f_0)$, then $f|_{\widetilde{W}}\in \mathcal{Q}_{\widetilde{W}}$ and $\mathcal{H}_*(f|_{\widetilde{W}}, \widetilde{\mathcal{V}})$ is an embedded complex analytic submanifold of $\mathcal{B}_{\widetilde{W}}$ of codimension $d(j_{W_0}f)$. Moreover, the complement of the closure of $\widetilde{W}$ in $\C$ is connected.

Let $\mathcal{V}$ be an arbitrary open subset of $\mathcal{V}_0$ which contains $f_0$.
Let us prove that for any $f\in \mathcal{V}$ with $d(j_{W_0}f)=d(j_{W_0}f_0)$,  $\mathcal{H}_*(f, \mathcal{V})$ is an embedded submanifold of $\mathcal{B}_W$ with codimension $d(j_{W_0}f)$ in a neighborhood of $f$.

By Mergelyan's Theorem, the natural inclusion map $\iota:\mathcal{B}_{W_0}\to \mathcal{B}_{\widetilde{W}}$ has dense image.  Let $\widetilde{E}:=\mathcal{H}_*(f|_{\widetilde{W}},\widetilde{\mathcal{V}})$ and $E:=\iota^{-1}(\widetilde{E})\cap \mathcal{V}$. By definition of an embedded submanifold, there is a holomorphic map $\psi$ from a neighborhood $\widetilde{\mathcal{U}}$ of $f|_{\widetilde{W}}$ in $\mathcal{B}_{\widetilde{W}}$ to $\C^d$, such that $\psi(0)=f|_{\widetilde{W}}$, $D\psi_g$ has rank $d$ for each $g\in \widetilde{\mathcal{U}}$ and $\psi^{-1}(0)=\widetilde{E}\cap \widetilde{\mathcal{U}}$. Let $\mathcal{U}=\mathcal{V}\cap \iota^{-1}(\widetilde{\mathcal{U}})$ which is an open set in $\mathcal{Q}_{W_0}$ containing $f_0$. Then $\Psi=\psi\circ \iota$ is a holomorphic map from $\mathcal{U}$ to $\C^d$. For any $g
\in \mathcal{U}$, $d\Psi_g=d\psi_{g|_{\widetilde{W}}} \circ \iota$. 
Since $\iota(\mathcal{B}_{W_0})$ is dense in $\mathcal{B}_{\widetilde{W}}$, $d\Psi_g(\mathcal{B}_W)$ is dense in $d\psi_{g|_{\widetilde{W}}}(\mathcal{B}_{\widetilde{W}})=\C^d$. Therfeore, $d\Psi_g(\mathcal{B}_W)=\C^d$. This shows that $\iota^{-1}(\widetilde{E})\cap \mathcal{U}$ is an embedded submanifold of $\mathcal{B}_W$ of codimension $d$. Shrinking $\mathcal{U}$ if necessary, we may assume that $E\cap \mathcal{U}=\iota^{-1}(\widetilde{E})\cap\mathcal{U}$ is connected. Then $E\cap \mathcal{U}\subset \mathcal{H}_*(f,\mathcal{V})\cap \mathcal{U}$. On the other hand,
$$E\cap\mathcal{U}=\iota^{-1}(\widetilde{E})\cap \mathcal{U}\supset \mathcal{H}_*(f, \mathcal{V})\cap \mathcal{U}.$$
Thus $\mathcal{H}_*(f,\mathcal{V})\cap \mathcal{U}=\iota^{-1}(\widetilde{E})\cap \mathcal{U}$ is an embedded complex submanifold of $\mathcal{B}_{W_0}$ of codimension $d$ near $f$. For each $g\in \mathcal{H}_*(f,\mathcal{V})$, $d(j_{W_0}g)=d(j_{W_0}f)=d(j_{W_0}f_0)=d$ and $\mathcal{H}_*(g,\mathcal{V})=\mathcal{H}_*(f,\mathcal{V})$. Thus the same proof shows that $\mathcal{H}_*(f,\mathcal{V})$ is  an embedded complex submanifold of $\mathcal{B}_{W_0}$ of codimension $d$ near $g$. 

Finally, let us show that $\mathcal{H}_*(f,\mathcal{V})$ is closed.  
Indeed, by Proposition~\ref{prop:connected set} it is a connected component of 
$\mathcal{H}_\#(f,\mathcal{V})$ and by Lemma~\ref{lem:closedH}, $\mathcal{H}_\#(f,\mathcal{V})$ is closed. 
\end{proof}

\begin{acknowledgement}
The work is supported by the National R\&D Program (Grant No. 2021YFA1003200). WS is supported by the New Cornerstone Science Foundation through the New Cornerstone Investigator Program and the XPLORER PRIZE.
YM was partially supported by NSFC (Grant No. 12501108), and
sponsored by the Natural Science Foundation of Shanghai (Grant No. 25ZR1402412).

The authors used AI(KIMI) to assist editing the paper and checking the proof. The counterexample in \S~\ref{subsec:counterex} was based on a suggestion of AI. 
\end{acknowledgement}

\bibliographystyle{amsplain}             

\end{document}